\documentclass[a4paper,12pt]{article}

\advance \textwidth -60truemm\relax

\advance\oddsidemargin -1truein\relax
\evensidemargin=\oddsidemargin

\advance\textheight -60truemm\relax
\advance\topmargin -1truein\relax
\unitlength\textwidth
\divide\unitlength by 200\relax

\usepackage{tikz}

\usepackage{amsmath,amssymb}
\usepackage{bm}

\bmdefine{\sss}{s}
\bmdefine{\vvv}{v}

\DeclareMathAlphabet{\mathscr}{U}{rsfs}{m}{n}

\newcommand{\msCCC}{\mathscr{C}}
\newcommand{\msOOO}{\mathscr{O}}
\newcommand{\msPPP}{\mathscr{P}}

\newcommand{\msIII}{\mathscr{I}}
\newcommand{\msFFF}{\mathscr{F}}
\newcommand{\msGGG}{\mathscr{G}}
\newcommand{\msXXX}{\mathscr{X}}

\newcommand{\NNN}{\mathbb{N}}
\newcommand{\ZZZ}{\mathbb{Z}}
\newcommand{\QQQ}{\mathbb{Q}}
\newcommand{\RRR}{\mathbb{R}}
\newcommand{\KKK}{\mathbb{K}}

\newcommand{\mmmm}{\mathfrak{m}}
\newcommand{\pppp}{\mathfrak{p}}

\newcommand{\RRRRR}{{\mathcal R}}
\newcommand{\SSSSS}{{\mathcal S}}
\newcommand{\TTTTT}{{\mathcal T}}

\newcommand{\FFFFF}{{\mathcal F}}
\newcommand{\GGGGG}{{\mathcal G}}

\newcommand{\TTTTTn}{{\mathcal T}^{(n)}}
\newcommand{\SSSSSn}{{\mathcal S}^{(n)}}

\newcommand{\covers}{\mathrel{\cdot\!\!\!>}}
\newcommand{\covered}{\mathrel{<\!\!\!\cdot}}
\newcommand{\coht}{\mathrm{coht}}
\newcommand{\height}{\mathrm{ht}}
\newcommand{\grade}{\mathrm{grade}}
\newcommand{\define}{\mathrel{:=}}
\newcommand{\defines}{\mathrel{=:}}

\newcommand{\gor}{Gorenstein}
\newcommand{\cm}{Cohen-Macaulay}

\newcommand{\rank}{\mathrm{rank}}

\newcommand{\relint}{{\rm{relint}}}

\newcommand{\linkcpx}{{\mathrm{link}}}
\newcommand{\starcpx}{{\mathrm{star}}}
\newcommand{\CR}{{\mathrm{CR}}}
\newcommand{\trace}{{\mathrm{tr}}}
\newcommand{\supp}{{\mathrm{supp}}}

\renewcommand{\hom}{{\mathrm{Hom}}}
\newcommand{\spec}{{\mathrm{Spec}}}

\newcommand{\aff}{\mathrm{aff}}

\newcommand{\condn}{{Condition N}}

\newcommand{\sscn}{{sequences with Condition N}}
\newcommand{\xip}{\xi^{+}}

\newcommand{\fc}{{\msFFF_{C}}}
\newcommand{\fcu}{{\msFFF_{(C_1,\ldots, C_u,C'_1, \ldots, C'_u)}}}
\newcommand{\lcu}{{L_{(C_1,\ldots, C_u,C'_1, \ldots, C'_u)}}}
\newcommand{\gab}{{\msGGG_{(a,b)}}}

\newcommand{\gabu}{{\msGGG_{(a_1,\ldots, a_u,b_1, \ldots, b_u)}}}
\newcommand{\mabu}{{M_{(a_1,\ldots, a_u,b_1, \ldots, b_u)}}}
\newcommand{\mab}{{M_{(a,b)}}}

\newcommand{\dist}{{\mathrm{dist}}}

\newcommand{\kcp}{{\KKK[\msCCC(P)]}}
\newcommand{\kop}{{\KKK[\msOOO(P)]}}
\newcommand{\ekp}{E_\KKK[\msPPP]}
\newcommand{\ekcp}{E_\KKK[\msCCC(P)]}
\newcommand{\ekop}{E_\KKK[\msOOO(P)]}

\newcommand{\cdm}{{C^{[d]}_\xi\cap C^{[m]}_\mu}}

\newcommand{\cdM}{{C^{[d]}_\xi\cap (C^{[M]}_\mu\cup C^{[M-1]}_\mu)}}

\newtheorem{thm}{Theorem}[section]
\newtheorem{fact}[thm]{Fact}
\newtheorem{example}[thm]{Example}
\newtheorem{lemma}[thm]{Lemma}
\newtheorem{cor}[thm]{Corollary}
\newtheorem{definition}[thm]{Definition}

\newtheorem{claim}{Claim}[thm]

\newcommand{\bigzerou}{\smash{\lower1.7ex\hbox{\bg 0}}}

\newcommand{\bigastu}{\smash{\lower1.7ex\hbox{\bg *}}}

\newcommand{\refeq}[1]{(\ref{#1})}
\numberwithin{equation}{section}

\newcommand{\cnp}{C}
\newcommand{\cp}{C'}

\newcommand{\mylabel}[1]{{\label{#1}\tt [#1]}}
\let\mylabel=\label

\title{%
Non-Gorenstein loci of Ehrhart rings of chain and order polytopes%
\footnote{%
This work was supported partially by the Research Institute for Mathematical Sciences, an International Joint Usage/Research Center located in Kyoto University.}
}

\author{\textsc{Mitsuhiro Miyazaki}{$^{1}$}%
\quad \textsc{and} \quad
\textsc{Janet Page}{$^{2}$}%
}

\date{\normalsize
}

\begin{document}

%\thispagestyle{empty}

%\setcounter{page}{0}
%\thispagestyle{empty}
%\clearpage
%\maketitle

\maketitle

\sloppy

\begin{abstract}
Let $P$ be a finite poset,
$\KKK$ a field,
 and $\msOOO(P)$ (resp.\ $\msCCC(P)$) the order
(resp.\ chain) polytope of $P$.
We study the non-\gor\ locus of $\ekop$ (resp.\ $\ekcp$),
the Ehrhart ring of $\msOOO(P)$ (resp.\ $\msCCC(P)$) over $\KKK$, 
which are each normal toric rings associated $P$.
In particular, we show that the dimension of non-\gor\ loci of $\ekop$
and $\ekcp$ are the same.
Further, we show that $\ekcp$ is nearly \gor\ if and only if
$P$ is the disjoint union of pure posets $P_1$, \ldots, $P_s$ with
$|\rank P_i-\rank P_j|\leq 1$ for any $i$ and $j$.
\end{abstract}

\section{Introduction}

In \cite{hhs}, Herzog, Hibi and Stamate studied the trace of  the canonical module
of a \cm\ local or graded ring:
for a commutative ring $R$ and an $R$-module $M$, the trace of $M$,
denoted by $\trace(M)$, is defined as follows. 
$$
\trace(M)\define\sum_{\varphi\in\hom(M,R)}\varphi(M).
$$
They showed among other things that the trace of the canonical module of a \cm\
local or graded ring is the defining ideal of non-\gor\ locus:
$R_\pppp$ is not \gor\ if and only if $\pppp\supset\trace(\omega_R)$ for $\pppp\in\spec(R)$,
where $\omega_R$ is the canonical module of $R$.  
In particular, we can see that $R$ is Gorenstein if and only if $\trace(\omega_R) = R$.
They also defined the notion of nearly Gorenstein:
a \cm\ local or graded ring $R$ is defined to be nearly \gor\ if $\trace(\omega_R)\supset\mmmm$,
where $\mmmm$ is the unique (graded) maximal ideal of $R$.
They studied nearly \gor\ property for various rings including Hibi rings:
a Hibi ring $\RRRRR_\KKK[\msIII(P)]$ is nearly \gor\ if and only if $P$
is the disjoint union of pure posets $P_1$, \ldots, $P_s$ with
$|\rank P_i-\rank P_j|\leq 1$ for any $i$ and $j$,
where $\msIII(P)$ is the set of poset ideals of $P$.

We study the trace of the canonical for two different toric rings associated to finite posets.  
Hibi began studying what are today known as Hibi rings \cite{hib}
in the study of the existence of algebras with straightening law (ASL for short)
which are integral domains.  
Around this time,
Stanley \cite{sta3} introduced the notion of order polytope $\msOOO(P)$ and chain polytope 
$\msCCC(P)$ of a finite poset $P$.
It was later observed that a Hibi ring on a finite poset $P$ is 
identical to the Ehrhart ring of the order polytope $\msOOO(P)$ of $P$.
Although the definition of order and chain polytopes are quite different,
they share many properties in common.
For example the Hilbert functions of the Ehrhart rings $\ekop$ and $\ekcp$ over a 
field $\KKK$
of the order and the chain polytopes of $P$ are the same and
$\ekop$ and $\ekcp$ have structures of ASL's on $\msIII(P)$ \cite{hib, hl}.

The first author studied the canonical modules of
$\ekop$ and $\ekcp$ and showed that the key notion to deal with the generators of the
canonical module of $\ekop$ (resp.\ $\ekcp$) is \sscn\ (resp. \condn') \cite{mo, mf, mc}.
Conditions N and N' are quite similar but N' is a little weaker than N.  
A similar condition also showed up as the study of ``mixed paths'' in \cite{pag1}.

In this paper, we study the non-\gor\ locus of the Ehrhart rings $\ekop$ and $\ekcp$ of
$\msOOO(P)$ and $\msCCC(P)$
in full generality.  
By the result of Herzog-Hibi-Stamate \cite{hhs}, studying the non-\gor\ locus of a ring is equivalent to studying the radical
of the trace of its canonical module.  Herzog-Mohammadi-Page \cite{hmp} studied this trace for certain toric rings and classified when some toric rings  (including Hibi rings) are Gorenstein on the punctured spectrum.
Here, we classify the trace of the canonical module more explicitly for $\ekop$ and $\ekcp$. We write the radical of the trace of the canonical module of each of these rings as an intersection of prime ideals which can be read off combinatorially from the poset.  More specifically, we show that if
there is a set of elements $a_1$, \ldots, $a_u$, $b_1$, \ldots,
$b_u$, $u\geq 1$, in $P\cup\{-\infty,\infty\}$ with
$a_1<b_1>a_2<b_2>\cdots>a_u<b_u>a_1$, 
$a_i\not\leq a_j$ (resp.\ $b_i\leq b_j$) for any $i$ and $j$ with $i\neq j$
 and
$\sum_{i=1}^u\rank([a_i,b_i])>\sum_{i=1}^{u-1}\dist(a_{i+1},b_i)+\dist(a_1,b_u)$,
then there is a 
prime ideal containing the trace of the canonical module
(which is not unique in the case of the chain polytope) 
corresponding to this set of elements.
See Section \ref{sec:pre} for 
the definitions of these symbols and notation.
Therefore, the dimension of the non-\gor\ loci of $\ekop$ or $\ekcp$ are
measured by the dimensions of these prime ideals.
In particular, we show that the dimensions of the non-\gor\ loci of $\ekop$ and $\ekcp$ are
the same.  
In either case, 
we note
that the dimension of the non-\gor\ locus is at most 
%$\dim\ekop-4$ (resp. $\dim\ekcp-4$).
the dimension of the ring minus $4$.
In particular, if $\pppp$ is a prime ideal of $\ekop$
(resp.\ $\ekcp$) with $\height\pppp\leq3$, then $\ekop_\pppp$ 
(resp.\ $\ekcp_\pppp$) is \gor.
We also show 
that we can construct $P$ so that the dimension of the non-Gorenstein locus of 
$\ekop$ or $\ekcp$ is $m$ for any $0 \leq m \leq n-4$, where $n = \dim \ekop = \dim \ekcp$.

This paper is organized as follows.
In Section \ref{sec:pre}, we establish notation and  recall some basic facts.
In Section \ref{sec:chain}, we study the condition that a given Laurent monomial is
contained in the radical of the trace of the canonical module of $\ekcp$.
We also show that $\ekcp$ is nealy \gor\ if and only if so is $\ekop$,
i.e., $\ekcp$ is nealy \gor\ if and only if $P$ is a disjoint union of pure posets
$P_1$, \ldots, $P_s$ with $|\rank P_i-\rank P_j|\leq 1$ for any $i$ and $j$, 
which mirrors the result of \cite{hhs} for $\ekop$.
In Section \ref{sec:order}, we study the condition that a given Laurent monomial
is contained in the radical trace of the canonical module of $\ekop$ by applying the result
established in Section \ref{sec:chain} to the ``covering relation poset''
$\CR(P)$ of $P$ (see Definition \ref{def:cr}).
Finally
in section 5, we study  
the prime ideals containing the trace of the canonical modules of
$\ekop$ and $\ekcp$
which appeared in the results of Sections \ref{sec:chain} and \ref{sec:order}
and 
show the correspondence of sets of elements mentioned above and 
these primes.
Although this correspondence is not one to one in the case of $\msCCC(P)$, 
we show that among prime ideals  
corresponding to sets 
of elements satisfying the condition above, the maximum of dimensions of 
these prime ideals is the dimension of the prime ideal that corresponds 
to the sequence above in the case of $\msOOO(P)$.
In particular, we
show that the dimensions of non-\gor\ loci of $\ekop$ and $\ekcp$ are the same.

\section{Preliminaries}

\mylabel{sec:pre}

In this paper, all rings and algebras are assumed to 
be commutative 
with an identity element 
unless stated otherwise.
%We also assume that a ring homomorphism maps the identity element to the identity element.
Also, when discussing properties of Noetherian rings, such as \cm\ or \gor\ properties,
we assume the rings under consideration are Noetherian.

First, we 
fix some notation we will use throughout the paper.
We denote the set of nonnegative integers, 
the set of integers, 
the set of rational numbers and 
the set of real numbers
%by $\RRR_{>0}$ the set of positive real numbers and
%by $\RRR_{\geq0}$ the set of nonnegative real numbers.
by $\NNN$, $\ZZZ$, $\QQQ$ and $\RRR$ respectively.
We denote the cardinality of a set $X$ by $\#X$.
For sets $X$ and $Y$, we define $X\setminus Y\define\{x\in X\mid x\not\in Y\}$.
%We use this notation even in the case where $Y\not\subset X$.
For nonempty sets $X$ and $Y$, we denote the set of maps from $X$ to $Y$ by $Y^X$.
If $X$ is a finite set, we identify $\RRR^X$ with the Euclidean space 
$\RRR^{\#X}$.
Note that by this identification, if $Y\subset X$, then the projection
$\RRR^X\to\RRR^Y$ corresponds to the restriction.
If $Y\subset X$,
we identify $\RRR^Y$ with $\{f\in\RRR^X\mid f(x)=0$ for $x\in X\setminus Y\}$. 
For a nonempty subset $\msXXX$ of $\RRR^X$, we denote by $\aff\msXXX$ the affine span of 
$\msXXX$, 
and by $\relint\msXXX$
the interior of $\msXXX$ in the topological space
$\aff\msXXX$.
Let $A$ be a subset of $X$.
We define the characteristic function $\chi_A\in\RRR^X$ by
$\chi_A(x)=1$ for $x\in A$ and $\chi_A(x)=0$ for $x\in X\setminus A$.
In order to clarify the domain of the characteristic function, we sometimes denote 
$\chi_A$ by $\chi_A^X$.
For $\xi$, $\xi'\in\RRR^X$ and $a\in \RRR$,
we define
maps $\xi\pm\xi'$ and $a\xi$ 
by
$(\xi\pm\xi')(x)=\xi(x)\pm\xi'(x)$ and
$(a\xi)(x)=a(\xi(x))$
for $x\in X$.
For $\xi\in\RRR^X$, we set $\supp\xi\define\{x\in X\mid\xi(x)\neq 0\}$.

Now we define a symbol which is frequently used in this paper.

\begin{definition}
\mylabel{def:xi+}
\rm
Let $X$ be a finite set and $\xi\in\RRR^X$.
For $B\subset X$, we set $\xip(B)\define\sum_{b\in B}\xi(b)$.
We define the empty sum to be 0, i.e., if $B=\emptyset$, then $\xi^+(B)=0$.
\end{definition}
We also define $\sum_{\ell=i}^ {i-1}a_\ell=0$.
Then $\sum _{\ell=i}^{j-1}a_\ell+\sum_{\ell=j}^{k-1}a_\ell=\sum_{\ell=i}^{k-1}a_\ell$ 
for any integers $i$, $j$ and $k$ with $i\leq j\leq k$.
We do not use symbol $\sum_{\ell=i}^j a_\ell$ if $j<i-1$.
Note that if $X$ is a finite set, $B$ is a subset of $X$ and $a\in\RRR$, then
$(\xi\pm\xi')^+(B)=\xip(B)\pm(\xi')^+(B)$
and $(a\xi)^+(B)=a(\xip(B))$.
%We denote $f\in\RRR^X$ with $f(x)=0$ for any $x\in X$ by 0.

Next we recall some definitions concerning finite partially
ordered sets (poset for short).
Let $Q$ be a finite poset.
(In the main argument, we use 
the symbol $P$ to indicate a poset but when discussing general
theory on posets, we use other symbols.)
We denote the set of maximal (resp.\ minimal) elements of $Q$
by $\max Q$ (resp.\ $\min Q$).
If $\max Q$ (resp.\ $\min Q$) consists of one element $z$, we often
abuse notation and write $z=\max Q$ (resp. $z=\min Q$).
A chain in $Q$ is a totally ordered subset of $Q$ or an empty set.
For a chain $C$ in $Q$, we define the length of $C$ as $\#C-1$.
The maximum length of chains in $Q$ is called the rank of $Q$ and denoted by $\rank Q$.
If every maximal (with respect to the inclusion relation) 
chain of $Q$ has the same length, we say that $Q$ is pure.
For a chain $C$ in $Q$, we set 
$$
\starcpx_Q(C)\define\{x\in Q\mid C\cup\{x\} \text{ is a
chain in }Q\}
$$
and
$$
\linkcpx_Q(C)\define\starcpx_Q(C)\setminus C.
$$ 
If either $a\leq b$ or $b\leq a$, we say that $a$ and $b$ are comparable and
denote 
this by $a\sim b$.
If $a$ and $b$ are not comparable, we say that $a$ and $b$ 
are incomparable and
denote 
this by $a\not\sim b$.
A subset $A$ of $Q$ is an antichain in $Q$ if every pair of two elements in $A$
are incomparable by the order of $Q$.
If a subset $I$ of $Q$ satisfies
%$I\subset Q$ and
$x\in I$, $y\in Q$, $y\leq x\Rightarrow y\in I$,
then we say that $I$ is a poset ideal of $Q$.

Let $\infty$ (resp.\ $-\infty$) be a new element which is not contained in $Q$.
We define a new poset $Q^+$ (resp.\ $Q^-$) whose base set is $Q\cup\{\infty\}$
(resp.\ $Q\cup\{-\infty\}$) and
$x<y$ in $Q^+$ (resp.\ $Q^-$) if and only if $x$, $y\in Q$ and $x<y$ in $Q$
or $x\in Q$ and $y=\infty$
(resp.\ $x=-\infty$ and $y\in Q$).
We set $Q^\pm\define(Q^+)^-$.
When treating multiple posets, we distinguish $\infty$ (resp.\ $-\infty$)
by adding subscript, like $\infty_Q$.

Let $Q'$ be an arbitrary poset. (We apply the following definition for
 $Q'=Q$, $Q^+$, $Q^-$ or $Q^\pm$.)
If $x$, $y\in Q'$, $x<y$ and there is no $z\in Q'$ with $x<z<y$,
we say that $y$ covers $x$ and denote by
$x\covered y$ or $y\covers x$.
For $x$, $y\in Q'$ with $x\leq y$, we set
$[x,y]_{Q'}\define\{z\in Q'\mid x\leq z\leq y\}$,
$[x,y)_{Q'}\define\{z\in Q'\mid x\leq z< y\}$ and
$(x,y]_{Q'}\define\{z\in Q'\mid x< z\leq y\}$.
Further, for $x$, $y\in Q'$ with $x<y$, we set
$(x,y)_{Q'}\define\{z\in Q'\mid x< z< y\}$.
We denote $[x,y]_{Q'}$ 
(resp.\ $[x,y)_{Q'}$, $(x,y]_{Q'}$ or $(x,y)_{Q'}$) 
as $[x,y]$ 
(resp.\ $[x,y)$, $(x,y]$ or $(x,y)$)
if there is no fear of confusion.
For $x\in Q$, we define the height (resp.\ coheight) of $x$ 
denoted by $\height x$ (resp. $\coht x$) by
$\height x\define \rank((-\infty,x])$
(resp.\ $\coht x\define\rank([x,\infty)$)).

\begin{definition}\rm
Let $Q'$ be an arbitrary finite poset and 
let $x$ and $y$ be elements of $Q'$ with $x\leq y$.
A saturated chain from $x$ to $y$ is a sequence of elements 
$z_0$, $z_1$, \ldots, $z_t$  of $Q'$ such that
$$
x=z_0\covered z_1\covered \cdots\covered z_t=y.
$$
\end{definition}
Note that the length of the chain $z_0$, $z_1$, \ldots, $z_t$ is $t$.

\begin{definition}\rm
Let $Q'$, $x$ and $y$ be as above.
We define 
$$\dist(x,y)\define\min\{t\mid \text{ there is a saturated chain from } x \text{ to } y
\text{ with length } t\}$$
and call $\dist(x,y)$ the distance of $x$ and $y$.
When treating multiple posets, we distinguish by adding a subscript, like $\dist_{Q'}$.
\end{definition}
Note that 
$\rank([x,y])=\max\{t\mid$ there is a saturated chain from $x$ to $y$ with length $t\}$.

Let $\KKK$ be a field and $R=\bigoplus_{n\in\NNN}R_n$ an
$\NNN$-graded $\KKK$-algebra with $R_0=\KKK$.
We denote the irrelevant maximal ideal $\bigoplus_{n>0}R_n$ by $\mmmm_R$
or simply $\mmmm$.
We call $\spec(R)\setminus\{\mmmm\}$ the punctured spectrum of $R$.

Let $R=\bigoplus_{n\in\NNN}R_n$ and $S=\bigoplus_{n\in\NNN}S_n$ be 
$\NNN$-graded $\KKK$-algebras with $R_0=S_0=\KKK$.
We denote by $R\#S$ the Segre product of $R$ and $S$:
$R\#S\define \bigoplus_{n\in\NNN}R_n\otimes_\KKK S_n$.

Next we fix notation about Ehrhart rings.
Let $X$ be a finite set with $-\infty\not\in X$ 
and $\msPPP$ a rational convex polytope in $\RRR^X$, i.e., 
a convex polytope whose vertices are contained in $\QQQ^X$.
Set $X^-\define X\cup\{-\infty\}$ and let $\{T_x\}_{x\in X^-}$ be
a family of indeterminates indexed by $X^-$.
For $f\in\ZZZ^{X^-}$, 
we denote the Laurent monomial 
$\prod_{x\in X^-}T_x^{f(x)}$ by $T^f$.
We set $\deg T_x=0$ for $x\in X$ and $\deg T_{-\infty}=1$.
Then the Ehrhart ring of $\msPPP$ over a field $\KKK$ is 
the toric ring corresponding to the cone over $\msPPP$.  
Namely, it is the $\NNN$-graded subring
$$
\KKK[T^f\mid f\in \ZZZ^{X^-}, f(-\infty)>0, \frac{1}{f(-\infty)}f|_X\in\msPPP]
$$
of the Laurent polynomial ring $\KKK[T_x^{\pm1}\mid x\in X][T_{-\infty}]$.
We denote the Ehrhart ring of $\msPPP$ over $\KKK$ by $\ekp$.
(We use $-\infty$ as the degree indicating element in order to be consistent with the
case of Hibi ring.)
If $\ekp$ is a standard graded algebra, i.e., generated as a $\KKK$-algebra by degree $1$
elements, we denote $\ekp$ by $\KKK[\msPPP]$.

It is known that the dimension (Krull dimension) of $\ekp$ is equal to $\dim \msPPP+1$.

Note that $\ekp$ is Noetherian since $\msPPP$ is rational,
and therefore it is normal by the criterion of Hochster \cite{hoc}.
Further, 
by the description of the canonical module of a normal affine semigroup ring
by Stanley \cite[p.\ 82]{sta2}, we see that the ideal
$$\bigoplus_{f\in\ZZZ^{X^-}, f(-\infty)>0, \frac{1}{f(-\infty)}f|_X\in\relint\msPPP}
\KKK T^f
$$
of $\ekp$ is the canonical module of $\ekp$.
%where $\relint\msPPP$ denotes the interior of $\msPPP$ in the topological space
%$\aff\msPPP$.
We call this ideal the canonical ideal of $\ekp$.

Note that if $X_1$ and $X_2$ are disjoint finite sets and 
$\msPPP_1$ and $\msPPP_2$ are rational convex polytopes in 
$\RRR^{X_1}$ and $\RRR^{X_2}$ respectively, then $\msPPP_1\times\msPPP_2$
is a rational convex polytope in $\RRR^{X_1\cup X_2}$.
Moreover, $E_\KKK[\msPPP_1\times\msPPP_2]=E_\KKK[\msPPP_1]\#E_\KKK[\msPPP_2]$.

\medskip

From now on, we fix a finite poset $P$.
First we recall the definitions of order and chain polytopes.
% \cite{sta3}.
\begin{definition}[{\cite{sta3}}]
\rm
%Let $P$ be a finite poset.
We set
$$
\msOOO(P)\define
\left\{f\in\RRR^P\left|\ \vcenter{\hsize=.5\textwidth\relax\noindent
$0\leq f(x)\leq 1$ for any $x\in P$ and
if $x<y$ in $P$, then $f(x)\geq f(y)$}\right.\right\}
$$
and
$$
\msCCC(P)\define
\left\{f\in\RRR^P\left|\ \vcenter{\hsize=.5\textwidth\relax\noindent
$0\leq f(x)$ for any $x\in P$ and 
$f^+(C)\leq 1$ for any chain in $P$}\right.\right\}.
$$
$\msOOO(P)$ (resp.\ $\msCCC(P)$) is called the order (resp.\ chain) polytope
of $P$.
\end{definition}

The Ehrhart ring $\ekop$ of the order polytope of $P$ is identical with
the ring 
$\RRRRR_\KKK[\msIII(P)]$
considered by Hibi \cite{hib}, which is nowadays called the Hibi ring,
where $\msIII(P)$ is the lattice of poset ideals of $P$ ordered by inclusion.
In particular, $\ekop$ is a standard graded $\KKK$-algebra.
It is also known that the Ehrhart ring $\ekcp$ of the chain polytope is a
standard graded $\KKK$-algebra (see. e.g., \cite[Proposition 2.9]{mc}).

Suppose that $P$ is a disjoint union of posets $P_1$ and $P_2$, i.e.,
$P=P_1\cup P_2$, $P_1\cap P_2=\emptyset$ and
$x<y$ in $P$ if and only if there is $i$ with $x$, $y\in P_i$ and
$x<y$ in $P_i$.
Then $\msOOO(P)=\msOOO(P_1)\times\msOOO(P_2)$
(resp.\ $\msCCC(P)=\msCCC(P_1)\times\msCCC(P_2)$).
In particular, 
$\kop=\KKK[\msOOO(P_1)]\#\KKK[\msOOO(P_2)]$
(resp.\ $\kcp=\KKK[\msCCC(P_1)]\#\KKK[\msCCC(P_2)]$).

Next we make the following definition (cf. \cite{mc,mf}).

\begin{definition}
\rm
Let $n\in\ZZZ$.
We set
$$
\TTTTTn(P)\define
\left\{\nu\in\ZZZ^{P^-}\left|\
\vcenter{\hsize=.5\textwidth\relax\noindent
$\nu(z)\geq n$ for any $z\in\max P$ and
if $x\covered y$ in $P^-$, then $\nu(x)\geq\nu(y)+n$}\right.\right\}
$$
and
$$
\SSSSSn(P)\define
\left\{\xi\in\ZZZ^{P^-}\left|\
\vcenter{\hsize=.5\textwidth\relax\noindent
$\xi(x)\geq n$ for any $x\in P$ and 
$\xi(-\infty)\geq\xi^+(C)+n$ for any maximal chain $C$ in $P$}\right.\right\}.
$$
If there is no fear of confusion, we abbreviate $\TTTTT^{(n)}(P)$
(resp.\ $\SSSSS^{(n)}(P)$) as $\TTTTT^{(n)}$ (resp.\ $\SSSSS^{(n)}$).
Further, we set $\nu(\infty)=0$ and extend $\nu\in\TTTTTn$ 
(resp.\ $\xi\in\SSSSSn$) to $\ZZZ^{P^\pm}$.
\end{definition}
We have the following.

\begin{fact}[{\cite{mc,mf}}]
\mylabel{fact:mc,mf}
\begin{eqnarray*}
\kop&=&\bigoplus_{\nu\in\TTTTT^{(0)}(P)}\KKK T^\nu,\\
\omega_\kop&=&\bigoplus_{\nu\in\TTTTT^{(1)}(P)}\KKK T^\nu,
\quad
\omega^{-1}_\kop=\bigoplus_{\nu\in\TTTTT^{(-1)}(P)}\KKK T^\nu,\\
\kcp&=&\bigoplus_{\xi\in\SSSSS^{(0)}(P)}\KKK T^\xi,\\
\omega_\kcp&=&\bigoplus_{\xi\in\SSSSS^{(1)}(P)}\KKK T^\xi,
\quad\mbox{and}\quad
\omega^{-1}_\kcp=\bigoplus_{\nu\in\SSSSS^{(-1)}(P)}\KKK T^\xi,\\
\end{eqnarray*}
%where $I^{-1}=R:_{Q(R)}I$ for a ring $R$ with total quotient ring $Q(R)$
where $I^{-1}\define\{x\in{Q(R)}\mid xI\subset R\}$ for a ring $R$ with total quotient ring $Q(R)$
and an ideal $I$ of $R$.
\end{fact}

We also make the following definition.

\begin{definition}
\rm
Let $\xi\in \ZZZ^P$, 
$\xi\in \ZZZ^{P^-}$, 
$\xi\in \ZZZ^{P^+}$ or
$\xi\in \ZZZ^{P^\pm}$ and $n\in \ZZZ$.
We set
$C_\xi^{[n]}\define\{ C\mid C$ is a maximal chain in $P$ and 
$\xip(C)=n\}$.
\end{definition}
Note that if $\xi\in\SSSSS^{(m)}$ and $\xi(-\infty)=d$, then $C_\xi^{[n]}=\emptyset$
 for any $n$ with $n>d-m$.

Now we recall the following.

\begin{fact}[{\cite[\S3 d)]{hib}}]
\mylabel{fact:order gor}
$\kop$ is \gor\ if and only if $P$ is pure.
\end{fact}
By using this fact and results of Stanley we see the following.

\begin{thm}
\mylabel{thm:chain gor}
$\kcp$ is \gor\ if and only if $P$ is pure.
\end{thm}
\begin{proof}
By \cite[Theorem 4.1]{sta3}, we see that the Hilbert functions of $\kcp$ and
$\kop$ are identical.
Therefore, by \cite[Theorem 4.4]{sta2}, we see that $\kcp$ is \gor\
if and only if so is $\kop$.
Further, we see by Fact \ref{fact:order gor} that $\kop$ is \gor\ if and only if
$P$ is pure.
Thus, we see the result.
\end{proof}

Next we note the following fact whose proof is easy but very
useful in the argument of chain polytopes.

\begin{lemma}
\mylabel{lem:cross chain}
Let $\nu\in\RRR^P$ and $D$ a nonempty set consisting of chains in $P$.
Set $M\define\max\{\nu^+(C)\mid C\in D\}$
and $m\define\min\{\nu^+(C)\mid C\in D\}$.
If $C_1$, $C_2\in D$,
$\nu^+(C_1)=\nu^+(C_2)=m$
(resp.\ $\nu^+(C_1)=\nu^+(C_2)=M$),
$z\in C_1\cap C_2$
and
$C_3\define((-\infty,z]\cap C_1)\cup((z,\infty)\cap C_2)$,
$C_4\define((-\infty,z]\cap C_2)\cup((z,\infty)\cap C_1)$
are elements of $D$, then
$$
\nu^+(C_3)=\nu^+(C_4)=m
$$
(resp.\ $\nu^+(C_3)=\nu^+(C_4)=M$),
\begin{eqnarray*}
&&\nu^+((-\infty,z]\cap C_1)=\nu^+((-\infty,z]\cap C_2)\\
\mbox{and}\\
&&\nu^+([z,\infty)\cap C_1)=\nu^+([z,\infty)\cap C_2).
\end{eqnarray*}
\end{lemma}
\begin{proof}
We first prove the case where $\nu^+(C_1)=\nu^+(C_2)=m$.
Since $C_3$, $C_4\in D$, we see that
\begin{eqnarray}
&&
\nu^+((-\infty,z]\cap C_1)+\nu^+((z,\infty)\cap C_2)=\nu^+(C_3)\geq m
\nonumber
\\
\mbox{and}
\mylabel{eq:c3c4geq}
\\
&&
\nu^+((-\infty,z]\cap C_2)+\nu^+((z,\infty)\cap C_1)=\nu^+(C_4)\geq m.
\nonumber
\end{eqnarray}
Thus,
$\nu^+((-\infty,z]\cap C_1)+\nu^+((z,\infty)\cap C_2)+
\nu^+((-\infty,z]\cap C_2)+\nu^+((z,\infty)\cap C_1)\geq 2m$.
On the other hand, by the assumption, we see that
\begin{eqnarray*}
&&
\nu^+((-\infty,z]\cap C_1)+\nu^+((z,\infty)\cap C_1)=\nu^+(C_1)= m
\\
\mbox{and}
\\
&&
\nu^+((-\infty,z]\cap C_2)+\nu^+((z,\infty)\cap C_2)=\nu^+(C_2)= m.
\end{eqnarray*}
Therefore,
$\nu^+((-\infty,z]\cap C_1)+\nu^+((z,\infty)\cap C_2)+
\nu^+((-\infty,z]\cap C_2)+\nu^+((z,\infty)\cap C_1)= 2m$.
Thus, we see that equalities hold in the ineqations \refeq{eq:c3c4geq},
\begin{eqnarray*}
&&\nu^+((-\infty,z]\cap C_1)=\nu^+((-\infty,z]\cap C_2)\\
\mbox{and}\\
&&\nu^+((z,\infty)\cap C_1)=\nu^+((z,\infty)\cap C_2).
\end{eqnarray*}
Thus the result follows.
The case where $\nu^+(C_1)=\nu^+(C_2)=M$ is proved similarly.
\end{proof}
We also note the following fact.

\begin{lemma}
\mylabel{lem:sum0}
Let $\eta\in \SSSSS^{(1)}$ and $\zeta\in\SSSSS^{(-1)}$.
%Suppose that $\eta+\zeta\in\SSSSS^{(0)}$.
Then for any $z\in P$ with $(\eta+\zeta)(z)=0$,
it holds that 
$\eta(z)=1$ and $\zeta(z)=-1$.
\end{lemma}
\begin{proof}
Since $\eta\in\SSSSS^{(1)}$ and $\zeta\in\SSSSS^{(-1)}$,
we see that $\eta(z)\geq 1$ and $\zeta(z)\geq -1$.
By assumption, we see that 
$\eta(z)+\zeta(z)
=(\eta+\zeta)(z)=0$.
Thus we see the result.
\end{proof}

Finally, we recall the notion of the trace of a module.

\begin{definition}
\rm
Let $R$ be a ring and $M$ an $R$-module.
We set
$$
\trace(M)\define\sum_{\varphi\in\hom(M,R)}\varphi(M)
$$
and call the trace of $M$.
\end{definition}
Note that if $M\cong M'$, then $\trace(M)=\trace(M')$.
For an ideal with positive grade, we have the following.

\begin{fact}[{\cite[Lemma 1.1]{hhs}}]
\mylabel{fact:trace of an ideal}
If $I$ is an ideal of a Noetherian ring $R$ with $\grade I>0$, then
$\trace(I)=I^{-1}I$.
\end{fact}

The trace of the canonical module is especially important by the following
fact.

\begin{fact}[{\cite[Lemma 2.1]{hhs}}]
\mylabel{fact:nongor locus}
Let $R$ be a \cm\ local ring with a canonical module $\omega_R$ or an
$\NNN$-graded \cm\ ring with $R_0$ is a field and canonical module $\omega_R$.
Then for $\pppp\in\spec(R)$,
$$
R_\pppp\mbox{ is \gor}\iff\pppp\not\supset\trace(\omega_R).
$$
\end{fact}
This fact shows that in the situation of Fact \ref{fact:nongor locus},
$V(\trace(\omega_R))\define\{\pppp\in\spec(R)\mid\pppp\supset\trace(\omega_R)\}$
is the non-\gor\ locus of $\spec(R)$,
i.e.,
$\trace(\omega_R)$ is a defining ideal of the non-\gor\ locus of $\spec(R)$.
Moreover, since $\omega_R$ is isomorphic to a height 1 ideal of $R$,
we can use Fact \ref{fact:trace of an ideal} to compute $\trace(\omega_R)$.
For a Noetherian ring $R$ and a prime ideal $\pppp$ in $R$, 
we denote $\dim R/\pppp$ by $\coht\pppp$, i.e., coheight of $\pppp$
in the poset of all prime ideals of $R$.
By Fact \ref{fact:nongor locus}, we see that non-\gor\ locus of a 
\cm\ local or graded ring is a closed subset of 
$\spec(R)$ with dimension $\max\{\coht\pppp\mid\pppp\supset\trace(\omega_R)\}$.

%%%%%%%%%%%%%%%%%%%%%%%%%%%%%%%%%%%%%%%%%%%%%%

\section{The case of chain polytopes}

\mylabel{sec:chain}

In this section, we consider the trace of the canonical module of 
the Ehrhart ring $\kcp$ of the chain polytope of $P$.
First we note the following fact.

\begin{lemma}
\mylabel{lem:nonpure link}
Let $\xi\in\SSSSS^{(0)}$ and $d=\xi(-\infty)$.
If there exists a chain $C$ in $P$ such that
$\xi^+(C)=d$ and $\starcpx_P(C)$ is not pure,
then
$$
T^\xi\not\in\sqrt{\trace(\omega_{\kcp})}.
$$
\end{lemma}
\begin{proof}
Assume the contrary.
Then, by Facts \ref{fact:mc,mf} and \ref{fact:trace of an ideal},
there exist a positive integer $N$,
$\eta\in\SSSSS^{(1)}$ and $\zeta\in\SSSSS^{(-1)}$ such that
$\eta+\zeta=N\xi$.

Let 
$x_1<x_2<\cdots<x_t$ and $y_1<y_2<\cdots<y_s$ be maximal chains in 
$\linkcpx_P(C)$ with $s<t$.
Then
$C_1\define C\cup\{x_1$, \ldots, $x_t\}$
and
$C_2\define C\cup\{y_1$, \ldots, $y_s\}$
are maximal chains in $P$.
Since $\xi(x_i)=\xi(y_j)=0$ for any $i$ and $j$,
we see by Lemma \ref{lem:sum0}, that
$\eta(x_i)=\eta(y_j)=1$ and $\zeta(x_i)=\zeta(y_j)=-1$ for any $i$ and $j$.
Therefore,
$\eta^+(C_1)=\eta^+(C)+t$ and $\zeta^+(C_2)=\zeta^+(C)-s$.
Since $C_1$ and $C_2$ are maximal chains in $P$ and $\eta\in\SSSSS^{(1)}$
and $\zeta\in\SSSSS^{(-1)}$, we see that
$\eta^+(C_1)+1\leq\eta(-\infty)$ and
$\zeta^+(C_2)-1\leq\zeta(-\infty)$.
Therefore,
$$
\eta^+(C)+\zeta^+(C)+t-s\leq\eta(-\infty)+\zeta(-\infty)=N\xi(-\infty)=Nd.
$$

On the other hand, since
$$
\eta^+(C)+\zeta^+(C)=N\xi^+(C)=Nd,
$$
we see that
$Nd+t-s\leq Nd$.
This is a contradiction.
\end{proof}

By this lemma, we see the following fact.

\begin{thm}
\mylabel{thm:pspec gor}
$\kcp$ is \gor\ on the punctured spectrum if and only if $P$ is a disjoint union
of pure posets.
In particular, $\kcp$ is nearly \gor\ if and only if $P$ is the disjoint union
of pure connected posets $P_1$, \ldots, $P_t$ with
$|\rank P_i-\rank P_j|\leq 1$ for any $i$ and $j$.
\end{thm}
\begin{proof}
Suppose that there is a nonempty chain $C$ in $P$ such that $\starcpx_P(C)$ is not
pure.
Set
$$
\xi(z)=
\left\{
\begin{array}{ll}
1,&\quad z\in C,\\
0,&\quad z\in P\setminus C,\\
\#C,&\quad z=-\infty.
\end{array}
\right.
$$
Then, $\xi\in\SSSSS^{(0)}$.
Moreover, $T^\xi\not\in\sqrt{\trace(\omega_\kcp)}$ by Lemma \ref{lem:nonpure link}.
Since $\xi(-\infty)=\#C>0$, we see that $T^\xi\in\mmmm$.
Thus, we see that $\mmmm\not\subset\sqrt{\trace(\omega_{\kcp})}$.

By contraposition, we see that if $\kcp$ is \gor\ on the punctured spectrum, then $\starcpx_P(C)$ 
is pure for any nonempty chain $C$ in $P$.
By \cite[Lemma 5.2]{hhs}, we see that if $\starcpx_P(C)$ is pure for any nonempty chain $C$
in $P$, $P$ is a disjoint union of pure posets.
Conversely, if $P$ is a disjoint union of pure posets, then by Theorem \ref{thm:chain gor}
and \cite[Theorem 4.15]{hhs}, we see that $\kcp$ is \gor\ on the punctured spectrum.
Thus, we see the first part.

Moreover, by \cite[Theorem 4.15]{hhs}, we see that $\kcp$ is nearly \gor\ if and
only if 
$P$ is the disjoint union
of pure connected posets $P_1$, \ldots, $P_t$ with
$|\rank P_i-\rank P_j|\leq 1$ for any $i$ and $j$.
\end{proof}

Next we show the following fact.

\begin{lemma}
\mylabel{lem:cycle}
Let $\xi\in\SSSSS^{(0)}$ and $d=\xi(-\infty)$.
If there are nonempty chains
$C_1$, \ldots, $C_u$, $C'_1$, \ldots, $C'_u$ in $P$
such that
$\max C_1<\min C'_1>\max C_2<\min C'_2>\cdots>\max C_u<\min C'_u
>\max C_1$,
$\sum_{\ell=1}^{u}\rank([\max C_\ell,\min C'_\ell])
>\sum_{\ell=1}^{u-1}\dist(\max C_{\ell+1},\min C'_\ell)
+\dist(\max C_1,\min C'_u)$
and
$\xi^+(C_i)+\xi^+(C'_j)=d$ for any $i$ and $j$,
then
$$
T^\xi\not\in\sqrt{\trace(\omega_{\kcp})}.
$$
\end{lemma}

\begin{proof}
Assume the contrary.
Then there exist a positive integer $N$, $\eta\in\SSSSS^{(1)}$
and $\zeta\in\SSSSS^{(-1)}$ such that $\eta+\zeta=N\xi$.

Take a maximal chain $C''_\ell$ in $(-\infty,\max C_\ell]$
with $C''_\ell\supset C_\ell$ for $1\leq\ell\leq d$.
Then $\xi^+(C''_\ell)\geq\xi^+(C_\ell)$ since $\xi\in\SSSSS^{(0)}$
and therefore, we see that
$\xi^+(C''_\ell)+\xi^+(C'_\ell)\geq d$.
On the other hand, we see that
$\xi^+(C''_\ell)+\xi^+(C'_\ell)=\xi^+(C''_\ell\cup C'_\ell)
\leq\xi(-\infty)=d$.
Therefore, we see that
$\xi^+(C''_\ell)=\xi^+(C_\ell)$.
Moreover, $\max C''_\ell=\max C_\ell$ by the choice of $C''_\ell$.
Thus, by replacing $C_\ell$ by $C''_\ell$, we may assume that
$C_\ell$ is a maximal chain in $(-\infty, \max C_\ell]$
for $1\leq\ell\leq u$.
We may also assume that $C'_\ell$ is a maximal chain in 
$[\min C'_\ell,\infty)$ for  any $\ell$.

Take saturated chains
$$
\max C_\ell=x_{\ell 0}\covered x_{\ell 1}\covered\cdots\covered x_{\ell t_\ell}=\min C'_\ell
$$
for $1\leq\ell\leq u$,
$$
\max C_{\ell+1}=y_{\ell 0}\covered y_{\ell 1}\covered\cdots\covered y_{\ell s_\ell}=\min C'_\ell
$$
for $1\leq \ell\leq u-1$ and
$$
\max C_{1}=y_{u 0}\covered y_{u 1}\covered\cdots\covered y_{u s_u}=\min C'_u
$$
with
$t_\ell=\rank([\max C_\ell,\min C'_\ell])$ for $1\leq \ell\leq u$,
$s_\ell=\dist(\max C_{\ell+1},\min C'_\ell)$ for $1\leq \ell\leq u-1$
and
$s_u=\dist(\max C_{1},\min C'_u)$.

Since $C_\ell\cup\{x_{\ell 1},\ldots, x_{\ell t_\ell-1}\}\cup C'_\ell$
is a maximal chain in $P$
and $\xi^+(C_\ell)+\xi^+(C'_\ell)=d$,
we see that $\xi(x_{\ell i})=0$ for $1\leq i\leq t_\ell-1$.
Thus, by Lemma \ref{lem:sum0} we see that
$$
\eta^+(C_\ell)+t_\ell-1+\eta^+(C'_\ell)+1\leq\eta(-\infty)
$$
for $1\leq\ell\leq u$.
We also see that 
%$\xi(y_{\ell i})=0$ for $1\leq i\leq s_\ell-1$ and
$$
\zeta^+(C_{\ell+1})-s_\ell+1+\zeta^+(C'_\ell)-1\leq\zeta(-\infty)
$$
for $1\leq \ell\leq u-1$ and 
%$\xi(y_{ui})=0$ for $1\leq i\leq s_u-1$,
$$
\zeta^+(C_{1})-s_u+1+\zeta^+(C'_u)-1\leq\zeta(-\infty)
$$
by the same way.
Therefore,
\begin{eqnarray*}
&&
\sum_{\ell=1}^u(\eta^+(C_\ell)+\eta^+(C'_\ell)+t_\ell)
+\sum_{\ell=1}^{u-1}(\zeta^+(C_{\ell+1})+\zeta^+(C'_\ell)-s_\ell)\\
&&\qquad
+\zeta^+(C_1)+\zeta^+(C'_u)-s_u\\
&\leq&u(\eta(-\infty)+\zeta(-\infty))\\
&=&ud.
\end{eqnarray*}
On the other hand,
\begin{eqnarray*}
&&\sum_{\ell=1}^u(\eta^+(C_\ell)+\zeta^+(C_\ell)+\eta^+(C'_\ell)+\zeta^+(C'_\ell))\\
&=&\sum_{\ell=1}^u(\xi(C_\ell)+\xi(C'_\ell))\\
&=&ud.
\end{eqnarray*}
Thus, we see that
$$
\sum_{\ell=1}^u t_\ell\leq\sum_{\ell=1}^u s_\ell.
$$
This contradicts our assumption.
\end{proof}

\newcommand{\assumpa}{{\bf (A)}}
\newcommand{\assumpb}{{\bf (B)}}
\newcommand{\assumpastast}{{\bf (**)}}
By Lemmas \ref{lem:nonpure link} and \ref{lem:cycle}, 
under the major premises
$\xi\in\SSSSS^{(0)}$ and $\xi(-\infty)=d$,
we see that if 
$$
T^\xi\in\sqrt{\trace{\omega_{\kcp}}}
$$
then it holds that
\begin{enumerate}
\item[\hskip\labelsep\assumpa]
for any chain $C$ in $P$ with $\xi^+(C)=d$, $\starcpx_P(C)$ is pure
and
\item[\hskip\labelsep\assumpb]
for any chains
$C_1$, \ldots, $C_u$, $C'_1$, \ldots, $C'_u$ in $P$  such that
\begin{enumerate}
\item[\hskip\labelsep\assumpastast]
\mylabel{page:assumpastast}
$\max C_1<\min C'_1>\max C_2<\min C'_2>\cdots>\max C_u<\min C'_u>\max C_1$,
$\max C_i\not\leq\max C_j$ (resp.\ $\min C'_i\not\leq\min C'_j$) for any $i$ and $j$
with $i\neq j$,
$C_i$ (resp.\ $C'_i$) is a maximal chain in $(-\infty,\max C_i]$
(resp.\ $[\min C'_i,\infty)$)
for any $i$ and
$\xi^+(C_i)+\xi^+(C'_j)=d$ for any $i$ and $j$, 
\end{enumerate}
it holds that
$\sum_{\ell=1}^u\rank([\max C_\ell,\min C'_\ell])
%=\sum_{\ell=1}^{u-1}\rank([\max C_{\ell+1},\min C'_\ell])
%+\rank([\max C_1,\min C'_u])$.
=\sum_{\ell=1}^{u-1}\dist(\max C_{\ell+1},\min C'_\ell)
+\dist(\max C_1,\min C'_u)$.
\end{enumerate}
We will show the converse of this fact
under the major premises
$\xi\in\SSSSS^{(0)}$ and $\xi(-\infty)=d$.

In order to do this task, we have to construct $\eta\in\SSSSS^{(1)}$,
$\zeta\in\SSSSS^{(-1)}$ and a positive integer $N$ with $\eta+\zeta=N\xi$.
Before going into details, let us observe examples.

\begin{example}
\rm
\mylabel{ex:nx}
Let
$$
P=
\vcenter{\hsize=.4\textwidth\relax
\begin{picture}(55,50)
%\put(20,3){\makebox(0,0)[t]{$P_1\setminus\{x_0\}$}}

\put(5,20){\circle*{3}}
\put(5,30){\circle*{3}}
\put(5,40){\circle*{3}}
\put(25,10){\circle*{3}}
\put(25,20){\circle*{3}}
\put(25,30){\circle*{3}}
\put(25,40){\circle*{3}}
\put(45,10){\circle*{3}}
\put(45,20){\circle*{3}}

\put(5,20){\line(0,1){20}}
\put(25,10){\line(0,1){30}}
\put(45,10){\line(0,1){10}}
\put(5,40){\line(1,-1){20}}
\put(25,10){\line(2,1){20}}
\put(25,20){\line(2,-1){20}}

\put(3,40){\makebox(0,0)[r]{$b_1$}}
\put(27,40){\makebox(0,0)[l]{$b_2$}}
\put(3,30){\makebox(0,0)[r]{$e_1$}}
\put(27,30){\makebox(0,0)[l]{$e_2$}}
\put(3,20){\makebox(0,0)[r]{$a_1$}}
\put(27,20){\makebox(0,0)[l]{$a_2$}}
\put(47,20){\makebox(0,0)[l]{$a_3$}}
\put(23,8){\makebox(0,0)[tr]{$d_1$}}
\put(47,8){\makebox(0,0)[tl]{$d_2$}}

\end{picture}
}
$$
$\xi(a_3)=2$,
$\xi(a_i)=\xi(b_i)=\xi(d_i)=1$,
$\xi(e_i)=0$ for $i=1,2$ and
$\xi(-\infty)=3$.
Then $\xi\in\SSSSS^{(0)}$.
Note that \assumpa\ and \assumpb\ are satisfied.
Let $C_1\define\{d_1, a_2, b_1\}$,
$C_2\define\{d_2, a_2 , b_1\}$,
$C_3\define\{d_1, a_2, e_2, b_2\}$,
$C_4\define\{d_2, a_2, e_2, b_2\}$,
$C_5\define\{d_1, a_3\}$,
$C_6\define\{d_2, a_3\}$
and
$C_7\define\{a_1,e_1,b_1\}$.
Then $C_i\in C_\xi^{[3]}$ for $1\leq i\leq 6$.
Therefore, if $\eta+\zeta=N\xi$ for $\eta\in\SSSSS^{(1)}$,
$\zeta\in\SSSSS^{(-1)}$ and a positive integer $N$,
then
\begin{eqnarray*}
&&\eta^+(C_i)\leq\eta(-\infty)-1\\
&&\zeta^+(C_i)\leq\zeta(-\infty)+1\\
&&\eta^+(C_i)+\zeta^+(C_i)=N\xi^+(C_i)=3N\\
\mbox{and}\\
&&\eta(-\infty)+\zeta(-\infty)=N\xi(\infty)=3N
\end{eqnarray*}
for $1\leq i\leq 6$.
Thus, we see that 
\begin{eqnarray}
&&\eta^+(C_i)=\eta(-\infty)-1
\nonumber\\
\mbox{and}
%\mylabel{eq:cd sum}
\nonumber\\
&&\zeta^+(C_i)=\zeta(-\infty)+1
\nonumber
\end{eqnarray}
for $1\leq i\leq 6$.
Note that $C_7$ does not have to satisfy these equations
since $\eta^+(C_7)+\zeta^+(C_7)=2N$.

On the other hand, since $\xi(e_i)=0$ for $i=1,2$, we see that
$\eta(e_i)=1$, $\zeta(e_i)=-1$ for $i=1,2$ by Lemma \ref{lem:sum0}.
Thus, we have to adjust values of $\eta$ so that
\begin{eqnarray*}
&&\eta(b_1)=\eta(b_2)+1\\
&&\eta(d_1)=\eta(d_2)\\
\mbox{and}\\
&&\eta(a_3)=\eta(a_2)+\eta(b_1)
\end{eqnarray*}
Therefore, for example, if we set
$\eta(a_3)=4$, 
$\eta(b_1)=3$,
$\eta(b_2)=\eta(d_1)=\eta(d_2)=2$,
%$\eta(a_1)=\eta(b_2)=\eta(d_1)=\eta(d_2)=2$,
$\eta(a_1)=\eta(a_2)=\eta(e_1)=\eta(e_2)=1$,
$\eta(-\infty)=7$,
$\zeta(a_1)=\zeta(a_2)=1$,
%$\zeta(a_1)=\zeta(a_3)=\zeta(b_2)=\zeta(d_1)=\zeta(d_2)=0$,
$\zeta(a_3)=\zeta(b_2)=\zeta(d_1)=\zeta(d_2)=0$,
$\zeta(b_1)=\zeta(e_1)=\zeta(e_2)=-1$
and
$\zeta(-\infty)=-1$,
then
$\eta\in\SSSSS^{(1)}$, $\zeta\in\SSSSS^{(-1)}$
and
$\eta+\zeta=2\xi$.
\end{example}

This example shows that just to set
$$
\left\{
\begin{array}{rcl}
\eta(z)&=&\xi(z)+1,\\
\zeta(z)&=&-1\\
\end{array}
\right.
$$
for $z\in P$ is not enough.

Consider another example.

\begin{example}
\rm
\mylabel{ex:hex}
Let
$$
P=
\vcenter{\hsize=.4\textwidth\relax
\begin{picture}(55,50)
%\put(20,3){\makebox(0,0)[t]{$P_1\setminus\{x_0\}$}}

\put(5,10){\circle*{3}}
\put(5,20){\circle*{3}}
\put(5,30){\circle*{3}}
\put(5,40){\circle*{3}}

\put(35,10){\circle*{3}}
\put(35,40){\circle*{3}}

\put(65,10){\circle*{3}}
\put(65,20){\circle*{3}}
\put(65,30){\circle*{3}}
\put(65,40){\circle*{3}}
\put(25,20){\circle*{3}}
\put(45,30){\circle*{3}}

\put(5,10){\line(0,1){30}}
\put(65,10){\line(0,1){30}}
\put(5,10){\line(1,1){30}}
\put(5,40){\line(1,-1){30}}
\put(35,10){\line(1,1){30}}
\put(35,40){\line(1,-1){30}}

\put(5,8){\makebox(0,0)[t]{$a_1$}}
\put(35,8){\makebox(0,0)[t]{$a_2$}}
\put(65,8){\makebox(0,0)[t]{$a_3$}}
\put(5,42){\makebox(0,0)[b]{$b_1$}}
\put(35,42){\makebox(0,0)[b]{$b_2$}}
\put(65,42){\makebox(0,0)[b]{$b_3$}}
\put(3,20){\makebox(0,0)[r]{$d_1$}}
\put(27,20){\makebox(0,0)[l]{$d_2$}}
\put(67,20){\makebox(0,0)[l]{$d_3$}}
\put(3,30){\makebox(0,0)[r]{$e_1$}}
\put(43,30){\makebox(0,0)[r]{$e_2$}}
\put(67,30){\makebox(0,0)[l]{$e_3$}}

\end{picture}
}
$$
$\xi(a_i)=\xi(b_i)=1$,
$\xi(d_i)=\xi(e_i)=0$
for $i=1,2,3$ and $\xi(-\infty)=2$.
Then $\xi\in\SSSSS^{(0)}$.
Note that \assumpa\ and \assumpb\ are satisfied.
Also let
$C_1\define\{a_1,d_1,e_1,b_1\}$,
$C_2\define\{a_1,b_2\}$,
$C_3\define\{a_2, d_2, b_1\}$,
$C_4\define\{a_2,b_3\}$
$C_5\define\{a_3,e_2,b_2\}$
and
$C_6\define\{a_3,d_3,e_3,b_3\}$.
Then 
$C_i\in C_\xi^{[2]}$ for $1\leq i\leq 6$.

Thus, if $\eta\in\SSSSS^{(1)}$, $\zeta\in\SSSSS^{(-1)}$ and
there is a positive integer $N$ with $\eta+\zeta=N\xi$, then by 
the same argument as in Example \ref{ex:nx}, we see that
$$
\eta^+(C_i)=\eta(-\infty)-1,\quad
\zeta^+(C_i)=\zeta(-\infty)+1
$$
for $1\leq i\leq 6$.
Further, by Lemma \ref{lem:sum0}, 
$$
\eta(d_i)=\eta(e_i)=1,\quad
\zeta(d_i)=\zeta(e_i)=-1
$$
for $i=1,2,3$.

Thus, we have to adjust values of other points so that $\eta^+(C_i)$
is constant for $1\leq i\leq 6$.
In this case, it is enough, for example, to set
$\eta(a_1)=2$, $\eta(a_2)=3$, $\eta(a_3)=1$,
$\eta(b_1)=4$, $\eta(b_2)=6$, $\eta(b_3)=5$,
$\eta(-\infty)=9$,
$\zeta(a_1)=3$, $\zeta(a_2)=2$, $\zeta(a_3)=4$,
$\zeta(b_1)=1$, $\zeta(b_2)=-1$, $\zeta(b_3)=0$
and
$\zeta(-\infty)=1$.
Then 
$\eta\in\SSSSS^{(1)}$, $\zeta\in\SSSSS^{(-1)}$ 
and
$\eta+\zeta=5\xi$.
\end{example}

By the argument above, when $\xi$ with 
\assumpa\ and \assumpb\ above is given, it is necessary to construct
$\eta\in\SSSSS^{(1)}$ and $\zeta\in\SSSSS^{(-1)}$ so that
$\eta(z)=1$ (resp.\ $\zeta(z)=-1$) for any $z\in P\setminus\supp\xi$ and
$\eta^+(C)$ (resp.\ $\zeta^+(C)$) is constant for $C\in C_\xi^{[d]}$.
In order to do this task in general, not by intuition of a human,
our strategy is to construct an adjustment function $\mu\in\ZZZ^{P^-}$
such that
$\mu(z)=1$ for any $z\in P\setminus\supp\xi$,
$\mu^+(C)$ is constant for any $C\in C_\xi^{[d]}$ and set
$\eta=N\xi+\mu$, $\zeta=N\xi-\mu$ for sufficiently large $N$.

W construct such $\mu$ in the following.

\begin{lemma}
\mylabel{lem:adjust sum}
Let $\xi\in\SSSSS^{(0)}$ and $d=\xi(-\infty)$.
If \assumpa\ and \assumpb\ are satisfied,
then there exists $\mu\in\ZZZ^P$ such that $\mu(z)=1$ for any $z\in P\setminus\supp\xi$
and $\mu^+(C)=\mu^+(C')$ for any $C$, $C'\in C^{[d]}_\xi$.
\end{lemma}
\begin{proof}
First note that by assumption \assumpa, it holds that
$\rank([\max C_{\ell+1},\min C'_\ell])=\dist(\max C_{\ell+1},\min C'_\ell)$
for $1\leq i\leq u-1$ and
$\rank([\max C_1,\min C'_u])=\dist(\max C_1,\min C'_u)$
for chains satisfying \assumpastast\ of \assumpb.

If $C^{[d]}_\xi=\emptyset$, the result trivially holds.
Assume that $C^{[d]}_\xi\neq\emptyset$.
For $\mu\in\ZZZ^P$, we set
$M(\mu)\define\max\{\mu^+(C)\mid C\in C^{[d]}_\xi\}$
and
$m(\mu)\define\min\{\mu^+(C)\mid C\in C^{[d]}_\xi\}$.
We show that for a given $\mu\in\ZZZ^P$ with $\mu(z)=1$ for any
$z\in P\setminus \supp\xi$ and $M(\mu)>m(\mu)$, there exists 
$\mu'\in\ZZZ^P$ such that $\mu'(z)=1$ for any $z\in P\setminus\supp\xi$,
$M(\mu')\leq M(\mu)$, $m(\mu')\geq m(\mu)$ and
$\#(C^{[d]}_\xi\cap C^{[M(\mu)]}_{\mu'})
<\#(C^{[d]}_\xi\cap C^{[M(\mu)]}_{\mu})$.
Then the result follows by the double induction on $M(\mu)-m(\mu)$
and 
$\#(C^{[d]}_\xi\cap C^{[M(\mu)]}_{\mu})$
starting from
$$
\mu_0(z)=
\left\{
\begin{array}{ll}
1&\quad z\in P\setminus\supp\xi,\\
0&\quad z\in\supp\xi.
\end{array}
\right.
$$

Set $M\define M(\mu)$ and $m\define m(\mu)$.

We begin by outlining our argument.
Take $C\in C_\xi^{[d]}\cap C_\mu^{[M]}$ and
take an appropriate element $a_1\in C\cap\supp\xi$.
Decrease the value of $\mu$ at $a_1$ by 1.
If there is $C'\in C_\xi^{[d]}\cap C_\mu^{[m]}$
with $a_1\in C'$, then
take $b_1\in C'\cap\supp\xi$ with $b_1\neq a_1$
and increase the value of $\mu$ at $b_1$ by 1 for each such $C'$.
If there is $C''\in C_\xi^{[d]}\cap(C_\mu^{[M]}\cup C_\mu^{[M-1]})$
with $b_1\in C''$,
then take $a_2\in C''\cap\supp\xi$ with $a_2\neq b_1$ and decrease the value of $\mu$
at $a_2$ by 1 for each such $C''$ and so on.
The main point in the proof is that we can take such $a_1$, $b_1$, $a_2$, \ldots
appropriately so that no infinite loop occures.
Claim \ref{claim:not meet c0} below assures no such infinite loop
and the essence of the proof of Claim \ref{claim:not meet c0}
is Claim \ref{claim:not in s+1}.

Now we return to our proof.

Take $C_0\in C^{[d]}_\xi\cap C^{[M]}_\mu$ and
set $C_0\cap\supp\xi=\{c_1, c_2,\ldots, c_t\}$,
$c_1<c_2<\cdots <c_t$.

Since $C_0\cap\supp\xi$ is a chain in $P$ with $\xi^+(C_0\cap\supp\xi)=d$,
we see that $(-\infty,c_1)$, $(c_t,\infty)$ and $(c_i,c_{i+1})$ for 
$1\leq i\leq t-1$ are pure by \assumpa.
Further, $(-\infty,c_1)\cap\supp\xi=(c_t,\infty)\cap\supp\xi=
(c_i,c_{i+1})\cap\supp\xi=\emptyset$ for $1\leq i\leq t-1$,
since
$d=\max\{\xi^+(C)\mid C$ is a chain in $P\}$
and $\xi(z)\geq 0$ for any $z\in P$.

If there is no $C\in \cdm$ with $C\ni c_1$,
then it is enough to set
$$
\mu'(z)=
\left\{
\begin{array}{ll}
\mu(z)&\quad z\neq c_1,\\
\mu(z)-1&\quad z=c_1,
\end{array}
\right.
$$
since 
$(\mu')^+(C_0)=\mu^+(C_0)-1$,
$\mu^+(C)-1\leq (\mu')^+(C)\leq\mu^+(C)$ for any chain $C$ in $P$
and 
$(\mu')^+(C)=\mu^+(C)$ for any $C\in \cdm$.

Assume that there exists $C\in \cdm$ with $C\ni c_1$.
Then
$\mu^+((-\infty,c_1)\cap C)=\mu^+((-\infty,c_1)\cap C_0)$
for such $C$,
since $(-\infty,c_1)$ is pure, $(-\infty,c_1)\cap\supp\xi=\emptyset$ and
$\mu(z)=1$ for any $z\in P\setminus\supp\xi$.
In particular,
$\{\ell\mid$ there exists $C\in \cdm$ such that
$C\ni c_\ell$, $\mu^+((-\infty,c_\ell)\cap C)=\mu^+((-\infty,c_\ell)\cap C_0)\}$
is not an empty set.

Let $s$ be the maximum number of this set.

Before going further, let us look at an example.
Consider $P$ and $\xi$ of Example \ref{ex:nx}.
If
$$
\mu(z)\define
\left\{
\begin{array}{ll}
1&\quad\mbox{$z\in\{e_1, e_2\}$},\\
0&\quad\mbox{otherwise},
\end{array}
\right.
$$
then $M=1$, $m=0$ and
$$
C_\xi^{[3]}\cap C_\mu^{[1]}=\{C_3, C_4\}.
$$
In both cases where $C_0=C_3$ or $C_0=C_4$, $c_s=a_2$.
Note that if we begin with the following procedure  
from $d_1$, we will go into an infinite loop.

Now we return to our proof.
In the setting above, we make the following claim.

\begin{claim}
\mylabel{claim: s<t}
$s<t$.
\end{claim}
In fact, assume the contrary.
Then there exists $C\in \cdm$ with $C\ni c_t$ and $\mu^+((-\infty,c_t)\cap C)
=\mu^+((-\infty,c_t)\cap C_0)$.
Since $(c_t,\infty)$ is pure, $(c_t,\infty)\cap\supp\xi=\emptyset$ and
$\mu(z)=1$ for any $z\in P\setminus\supp\xi$, we see that
$\mu^+((c_t,\infty)\cap C)=\mu^+((c_t,\infty)\cap C_0)$.
Therefore, 
$$
m=\mu^+(C)=\mu^+(C_0)=M.
$$
This contradicts our assumption.

Next we
show the following fact which is essential to ensure there are no infinite loops.

\begin{claim}
\mylabel{claim:not in s+1}
There is no $C\in\cdm$ with $C\ni c_s$, $c_{s+1}$.
\end{claim}
In fact, assume the contrary and take such $C$.
Also take $C'\in\cdm$ with $C'\ni c_s$ and $\mu^+((-\infty,c_s)\cap C')
=\mu^+((-\infty,c_s)\cap C_0)$.
By applying Lemma \ref{lem:cross chain} 
first to $\xi$ and the set of all maximal chains in $P$ 
%first to $\xi$, the set of all maximal chains in $P$, $C$ and $C'$, 
we see that $((-\infty,c_s]\cap C)\cup ((c_s,\infty)\cap C')$, 
$((-\infty,c_s]\cap C')\cup((c_s,\infty)\cap C) \in C_\xi^{[d]}$
and next 
to $\mu$ and $C^{[d]}_\xi$,
%to $\mu$, $C^{[d]}_\xi$, $C$ and $C'$
we see that 
$\mu^+((-\infty,c_s)\cap C)=\mu^+((-\infty,c_s)\cap C')$.
Since $(c_s, c_{s+1})$ is pure, $(c_s, c_{s+1})\cap\supp\xi=\emptyset$
and $\mu(z)=1$ for any $z\in P\setminus\supp\xi$, we see that
$\mu^+((c_s,c_{s+1})\cap C)=\mu^+((c_s,c_{s+1})\cap C_0)$.
Therefore,
$\mu^+((-\infty,c_{s+1})\cap C)=\mu^+((-\infty,c_{s+1})\cap C_0)$.
This contradicts to the maximality of $s$.

\medskip

Set 
$$A_1\define\{c_s\}$$ and set
\begin{align*}
B_1\define \{b\in\supp\xi\mid &\text{ there exists } C\in\cdm \text{ such that } \\ &c_s, b\in C \text{  
and } c_s\covered b \text{ in } \supp\xi\}.
\end{align*}
For $i>1$ we define $A_i$ and $B_i$ inductively.
If $A_1$, \ldots, $A_{i-1}$, $B_1$, \ldots, $B_{i-1}$ are defined, we set
\begin{align*}
A_i\define\{a\in\supp\xi\setminus \bigcup_{j=1}^{i-1}A_j\mid&  \text{ 
there exists } b\in B_{i-1},C\in\cdM \\ & \text{ such that
}a, b\in C \text{ and } a\covered b \text{ in }\supp\xi\},
\end{align*}
\begin{align*}
B_i\define\{b\in\supp\xi\setminus \bigcup_{j=1}^{i-1}B_j\mid &\text{
there exists } a\in A_i\text{ and }C\in\cdm \text{ such that }
\\ & a, b\in C \text{ and } a\covered b\text{ in }\supp\xi\}.
\end{align*}
Also set
$A\define\bigcup_{i=1}^\infty A_i$ and $B\define\bigcup_{i=1}^\infty B_i$.

Before going further, let us look at an example.
Consider $P$ and $\xi$ of Example \ref{ex:hex}.
Since $\xi(-\infty)=2$, $d=2$.
Set
$$
\mu(z)\define
\left\{
\begin{array}{ll}
1&\quad z\in\{d_1,d_2, d_3,e_1,e_2,e_3\},\\
0&\quad z\in\{a_1,a_2, a_3,b_1,b_2,b_3\}.\\
\end{array}
\right.
$$
%$\xi(-\infty)=2$.
Then in the notation of Example \ref{ex:hex},
$C_i\in C_\xi^{[2]}$ for $1\leq i\leq 6$.
Thus, $M=2$ and $m=0$.
Further, $C_\xi^{[2]}\cap C_\mu^{[2]}=\{C_1, C_6\}$.
If we select $C_0$ as $C_1$, then $A_1=\{a_1\}$,
$B_1=\{b_2\}$, $A_2=\{a_3\}$ and $B_2=\emptyset$.
Thus, $A=\{a_1,a_3\}$ and $B=\{b_2\}$.
If we select $C_0$ as $C_6$, then this is a simple case noted at the 
beginning of the proof.

Now we continue our proof.
In this setting, we have the following.

\begin{claim}
\mylabel{claim:anti chain}
For any $a\in A$  (resp.\ $b\in B$) and $C\in C^{[d]}_\xi$ with $a\in C$
(resp.\ $b\in C$), it holds that
$\xi^+((-\infty,a]\cap C)=\xi^+((-\infty,c_s]\cap C_0)$
(resp.\ $\xi^+([b,\infty)\cap C)=d-\xi^+((-\infty,c_s]\cap C_0)$).
In particular, $A$ and $B$ are antichains.
\end{claim}
In fact, let $a\in A_1$.
Then $a=c_s$.
By applying Lemma \ref{lem:cross chain} 
%to $\xi$, the set of all maximal chains in $P$, $C$ and $C'$, 
to $\xi$ and the set of all maximal chains in $P$, 
we see that
$$
\xi^+((-\infty,a]\cap C)=\xi^+((-\infty,c_s]\cap C_0).
$$
Next,
let $b\in B_1$ and $C\in C^{[d]}_\xi$ with $b\in C$.
By the definition of $B_1$, we see that 
$c_s\covered b$ in $\supp\xi$ and
there is $C^{(1)}\in C^{[d]}_\xi$
with $c_s$, $b\in C^{(1)}$.
By applying Lemma \ref{lem:cross chain} 
to $\xi$ and the set of all maximal chains in $P$, 
%to $\xi$, the set of all maximal chains in $P$, $C$ and $C'$, 
we see that
$$
\xi^+((-\infty,c_s]\cap C_0)=\xi^+((-\infty,c_s]\cap C^{(1)}).
$$
Since $c_s\covered b$ in $\supp\xi$, we see that 
$$
\xi^+((-\infty,c_s]\cap C^{(1)})+\xi^+([b,\infty)\cap C^{(1)})=d
$$
and therefore
$$
\xi^+([b,\infty)\cap C^{(1)})=d-\xi^+((-\infty,c_s]\cap C_0).
$$
By applying Lemma \ref{lem:cross chain} 
to $\xi$ and the set of all maximal chains in $P$, 
%to $\xi$, the set of all maximal chains in $P$, $C$ and $C'$, 
we see that
$$
\xi^+([b,\infty)\cap C)=\xi^+([b,\infty)\cap C^{(1)})=d-\xi^+((-\infty,c_s]\cap C_0).
$$

Next, let $a\in A_2$ and $C\in C^{[d]}_\xi$ with $a\in C$.
Then by the definition of $A_2$, we see that there are $b\in B_1$ and 
$C^{(2)}\in C^{[d]}_\xi$ such that $a\covered b$ in $\supp\xi$ and
$a$, $b\in C^{(2)}$.
By applying Lemma \ref{lem:cross chain} 
to $\xi$ and the set of all maximal chains in $P$, 
%to $\xi$, the set of all maximal chains in $P$, $C$ and $C'$, 
we see that
$$
\xi^+((-\infty,a]\cap C)=\xi^+((-\infty,a]\cap C^{(2)}).
$$
Since $a\covered b$ in $\supp\xi$, we see that
$$
\xi^+((-\infty,a]\cap C^{(2)})+\xi^+([b,\infty)\cap C^{(2)})=d.
$$
Moreover, by the fact shown above, we see that
$$
\xi^+([b,\infty)\cap C^{(2)})=d-\xi^+((-\infty,c_s]\cap C_0).
$$
Thus, we see that
$$
\xi^+((-\infty,a]\cap C)=\xi^+((-\infty,c_s]\cap C_0).
$$

By repeating this argument and using induction, we see the first part of the assertion.
Next let $a$, $a'\in A$.
Take $C$, $C'\in C^{[d]}_\xi$ with $a\in C$ and $a'\in C'$.
By the first part of the assertion, we see that
$$
\xi^+((-\infty,a]\cap C)=\xi^+((-\infty,a']\cap C').
$$
If $a<a'$, take a maximal chain $C''$ containing
$((-\infty,a]\cap C)\cup([a',\infty)\cap C')$.
Then
\begin{eqnarray*}
\xi^+(C'')&\geq&
\xi^+((-\infty,a]\cap C)+\xi^+([a',\infty)\cap C')\\
&=&
\xi^+((-\infty,a]\cap C)+\xi(a')+\xi^+((a',\infty)\cap C')\\
&=&
\xi^+((-\infty,a]\cap C)+\xi(a')+\xi^+(C')-\xi^+((-\infty,a']\cap C')\\
&=&
\xi(a')+\xi^+(C')\\
&=&
d+\xi(a').
\end{eqnarray*}
%by the first part.
This contradicts the facts that $\xi\in \SSSSS^{(0)}$,
$\xi(-\infty)=d$, $a'\in A\subset\supp\xi$ and $C''$ is a maximal chain in $P$.
Thus, we see that $A$ is an antichain.
We also see that $B$ is an antichain by the same way.

\medskip

Now we

\begin{claim}
\mylabel{claim:not meet c0}
$C_0\cap B=\emptyset$.
\end{claim}
In fact, assume the contrary and let $B_u\cap C_0\neq\emptyset$.
Take $b_u\in B_u\cap C_0$.
Note that $u>1$.
In fact, if $u=1$, then $b_u\in B_1$ and therefore $b_u\covers c_s$ in $\supp\xi$.
Since $b_u\in C_0$, we see that $b_u=c_{s+1}$.
On the other hand, since $b_u\in B_1$, there is $C\in C_\xi^{[d]}\cap C_\mu^{[m]}$ 
with $c_s$, $b_u\in C$.
This would contradict Claim \ref{claim:not in s+1}.
Thus, $u>1$.

Take $a_u\in A_u$ and $\cnp_u\in\cdm$ such that $a_u$, $b_u\in \cnp_u$ and
$a_u\covered b_u$ in $\supp\xi$.
Take $b_{u-1}\in B_{u-1}$ and $\cp_{u-1}\in\cdM$ such that $a_u$, $b_{u-1}\in \cp_{u-1}$
and $a_u\covered b_{u-1}$ in $\supp\xi$.
Take $a_{u-1}\in A_{u-1}$ and $\cnp_{u-1}\in\cdm$ such that 
$a_{u-1}$, $b_{u-1}\in \cnp_{u-1}$ and $a_{u-1}\covered b_{u-1}$ in $\supp\xi$.
Take $b_{u-2}\in B_{u-2}$ and $\cp_{u-2}\in\cdM$ such that $a_{u-1}$, $b_{u-2}\in \cp_{u-2}$
and $a_{u-1}\covered b_{u-2}$ in $\supp\xi$.
And so on.
%Set $\cp_u\define C_0$.

By Claim \ref{claim:anti chain}, we see that
\begin{equation}
\xi^+((-\infty,a_i]\cap \cnp_i)+\xi^+([b_j,\infty)\cap \cnp_j)=d
\mylabel{eq:sum d}
\nonumber
\end{equation}
for any $i$ and $j$.
Since $a_i\neq a_j$ and $b_i\neq b_j$ if $i\neq j$,
and $A$ and $B$ are antichains, we see that chains
$(-\infty,a_1]\cap \cnp_1$,
$(-\infty,a_2]\cap \cnp_2$, \ldots,
$(-\infty,a_u]\cap \cnp_u$,
$[b_1,\infty)\cap \cnp_1$,
$[b_2,\infty)\cap \cnp_2$, \ldots,
$[b_u,\infty)\cap \cnp_u$
satisfy \assumpastast\  of \assumpb.
Moreover, since
$$
\xi^+((-\infty,c_s]\cap C_0)+\xi^+([b_u,\infty)\cap C_0)=d
$$
by Claim \ref{claim:anti chain}, we see that $b_u=c_{s+1}$.

Set $C''_i\define ((-\infty,a_{i+1}]\cap \cnp_{i+1})\cup((a_{i+1},b_i)\cap \cp_i)
\cup([b_i,\infty)\cap \cnp_i)$
for $1\leq i\leq u-1$
and
$C''_u\define ((-\infty,a_{1}]\cap \cnp_{1})\cup((a_{1},b_u)\cap C_0)
\cup([b_u,\infty)\cap \cnp_u)$.
Then
$C''_i\in C^{[d]}_\xi$ for $1\leq i\leq u$.
Therefore,
$\mu^+(C''_i)\geq m$ for $1\leq i\leq u-1$
and 
$\mu^+(C''_u)>m$ by Claim \ref{claim:not in s+1}.
Thus,
since $(a_1,b_u)$ and $(a_{i+1},b_i)$ for $1\leq i\leq u-1$ are pure by \assumpa\
applied to $C_0$ and $\cp_i$ for $1\leq i\leq u-1$
and $\mu(z)=1$ for any $z\in (a_1,b_u)\cup\bigcup_{i=1}^{u-1}(a_{i+1},b_i)$,
we see that
\begin{eqnarray*}
&&
\sum_{i=1}^u(\mu^+((-\infty,a_i]\cap \cnp_i)+\mu^+([b_i,\infty)\cap \cnp_i))\\
&&\qquad
+\sum_{i=1}^{u-1}\rank([a_{i+1},b_i))+\rank([a_1,b_u))\\
&=&\sum_{i=1}^u\mu^+(C''_i)\\
&>&mu.
\end{eqnarray*}
On the other hand, we see that
\begin{eqnarray*}
&&
\sum_{i=1}^u(\mu^+((-\infty,a_i]\cap \cnp_i)+\mu^+([b_i,\infty)\cap \cnp_i)
+\rank([a_i,b_i)))\\
&=&\sum_{i=1}^u\mu^+(\cnp_i)\\
&=&mu.
\end{eqnarray*}
by the same way.
Therefore, we see that
$$
\sum_{i=1}^{u-1}\rank([a_{i+1},b_i))+\rank([a_1,b_u))
>\sum_{i=1}^u\rank([a_i,b_i)).
$$
This contradicts \assumpb.

Thus, we see that $B\cap C_0=\emptyset$.

\medskip

Set
$$
\mu'(z)\define
\left\{
\begin{array}{ll}
\mu(z)&\quad z\in P\setminus(A\cup B),\\
\mu(z)-1&\quad z\in A,\\
\mu(z)+1&\quad z\in B.
\end{array}
\right.
$$
We claim that $m(\mu')\geq m$, $M(\mu')\leq M$ and
$\#(C^{[d]}_\xi\cap C^{[M]}_{\mu'})
<\#(C^{[d]}_\xi\cap C^{[M]}_{\mu})$.
First note that 
there is no element in $C^{[d]}_\xi$ which contains 2 or
more elements of $A$ (resp.\ $B$), since $A$ (resp.\ $B$) is an antichain.

Let $C\in C^{[d]}_\xi$.
If $C\cap(A\cup B)=\emptyset$ or
$C\cap A\neq\emptyset$ and $C\cap B\neq\emptyset$,
then
$(\mu')^+(C)=\mu^+(C)$,
by the fact noted above.
If $C\cap A\neq\emptyset $ and $C\cap B=\emptyset$, then
$(\mu')^+(C)=\mu^+(C)-1$.
%since $C$ does not contain 2 or more elements of $A$.
If $C\cap A=\emptyset $ and $C\cap B\neq\emptyset$, then
$(\mu')^+(C)=\mu^+(C)+1$.
%since $C$ does not contain 2 or more elements of $B$.

Suppose that $C\in\cdm$.
If $C\cap A\neq\emptyset$, then by the definition of $B$, we see that
$C\cap B\neq\emptyset$.
Therefore, $(\mu')^+(C)=m$.
Further, if $C\cap A=\emptyset$, we see by the above argument,
that $(\mu')^+(C)\geq m$.
Thus, we see that $m(\mu')\geq m$.
We see that for any $C\in \cdM$,
$(\mu')^+(C)\leq\mu^+(C)$ by the same way.
In particular, $M(\mu')\leq M$.
Moreover, since $C_0\cap A\neq\emptyset$ and $C_0\cap B=\emptyset$,
we see that $C_0\not\in C^{[M]}_{\mu'}$.
Since $(\mu')^+(C)\leq M-1$ for any $C\in C^{[d]}_\xi\cap C^{[M-1]}_\mu$,
we see that
$\#(C^{[d]}_\xi\cap C^{[M]}_{\mu'})
<\#(C^{[d]}_\xi\cap C^{[M]}_\mu)$.
\end{proof}

Let us observe the construction of Lemma \ref{lem:adjust sum} by
$P$ and $\xi$ of Examples \ref{ex:nx} and \ref{ex:hex}.
First consider $P$ and $\xi$ of Example \ref{ex:nx}.
$C_0$ in the proof of Lemma \ref{lem:adjust sum} is $C_1$ or $C_2$
and $c_s=a_2$.
Further, $A_1=\{a_2\}$, $B_1=\{b_1\}$ and $A_2=\emptyset$.
Thus,
$$
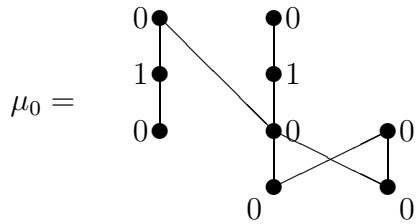
\mu_0=
\vcenter{\hsize=.4\textwidth\relax
\begin{picture}(55,50)
%\put(20,3){\makebox(0,0)[t]{$P_1\setminus\{x_0\}$}}

\put(5,20){\circle*{3}}
\put(5,30){\circle*{3}}
\put(5,40){\circle*{3}}
\put(25,10){\circle*{3}}
\put(25,20){\circle*{3}}
\put(25,30){\circle*{3}}
\put(25,40){\circle*{3}}
\put(45,10){\circle*{3}}
\put(45,20){\circle*{3}}

\put(5,20){\line(0,1){20}}
\put(25,10){\line(0,1){30}}
\put(45,10){\line(0,1){10}}
\put(5,40){\line(1,-1){20}}
\put(25,10){\line(2,1){20}}
\put(25,20){\line(2,-1){20}}

\put(3,40){\makebox(0,0)[r]{$0$}}
\put(27,40){\makebox(0,0)[l]{$0$}}
\put(3,30){\makebox(0,0)[r]{$1$}}
\put(27,30){\makebox(0,0)[l]{$1$}}
\put(3,20){\makebox(0,0)[r]{$0$}}
\put(27,20){\makebox(0,0)[l]{$0$}}
\put(47,20){\makebox(0,0)[l]{$0$}}
\put(23,8){\makebox(0,0)[tr]{$0$}}
\put(47,8){\makebox(0,0)[tl]{$0$}}

\end{picture}
}
$$
and
$$
\mu_1=
\vcenter{\hsize=.4\textwidth\relax
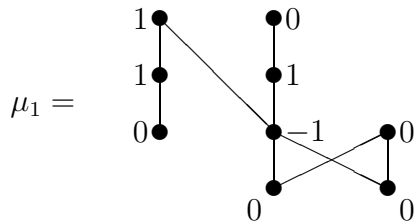
\begin{picture}(55,50)
%\put(20,3){\makebox(0,0)[t]{$P_1\setminus\{x_0\}$}}

\put(5,20){\circle*{3}}
\put(5,30){\circle*{3}}
\put(5,40){\circle*{3}}
\put(25,10){\circle*{3}}
\put(25,20){\circle*{3}}
\put(25,30){\circle*{3}}
\put(25,40){\circle*{3}}
\put(45,10){\circle*{3}}
\put(45,20){\circle*{3}}

\put(5,20){\line(0,1){20}}
\put(25,10){\line(0,1){30}}
\put(45,10){\line(0,1){10}}
\put(5,40){\line(1,-1){20}}
\put(25,10){\line(2,1){20}}
\put(25,20){\line(2,-1){20}}

\put(3,40){\makebox(0,0)[r]{$1$}}
\put(27,40){\makebox(0,0)[l]{$0$}}
\put(3,30){\makebox(0,0)[r]{$1$}}
\put(27,30){\makebox(0,0)[l]{$1$}}
\put(3,20){\makebox(0,0)[r]{$0$}}
\put(27,20){\makebox(0,0)[l]{$-1$}}
\put(47,20){\makebox(0,0)[l]{$0$}}
\put(23,8){\makebox(0,0)[tr]{$0$}}
\put(47,8){\makebox(0,0)[tl]{$0$}}

\end{picture}
}
$$
Note that, since the maximal chain $\{a_1, e_1, b_1\}$ does not belong
to $C_\xi^{[3]}$, we do not have to care about it.

Next consider $P$ and $\xi$ of Example \ref{ex:hex}.
First set
$$
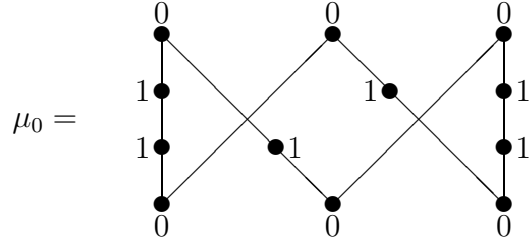
\mu_0=
\vcenter{\hsize=.4\textwidth\relax
\begin{picture}(55,50)

\put(5,10){\circle*{3}}
\put(5,20){\circle*{3}}
\put(5,30){\circle*{3}}
\put(5,40){\circle*{3}}

\put(35,10){\circle*{3}}
\put(35,40){\circle*{3}}

\put(65,10){\circle*{3}}
\put(65,20){\circle*{3}}
\put(65,30){\circle*{3}}
\put(65,40){\circle*{3}}
\put(25,20){\circle*{3}}
\put(45,30){\circle*{3}}

\put(5,10){\line(0,1){30}}
\put(65,10){\line(0,1){30}}
\put(5,10){\line(1,1){30}}
\put(5,40){\line(1,-1){30}}
\put(35,10){\line(1,1){30}}
\put(35,40){\line(1,-1){30}}

\put(5,8){\makebox(0,0)[t]{$0$}}
\put(35,8){\makebox(0,0)[t]{$0$}}
\put(65,8){\makebox(0,0)[t]{$0$}}
\put(5,42){\makebox(0,0)[b]{$0$}}
\put(35,42){\makebox(0,0)[b]{$0$}}
\put(65,42){\makebox(0,0)[b]{$0$}}
\put(3,20){\makebox(0,0)[r]{$1$}}
\put(27,20){\makebox(0,0)[l]{$1$}}
\put(67,20){\makebox(0,0)[l]{$1$}}
\put(3,30){\makebox(0,0)[r]{$1$}}
\put(43,30){\makebox(0,0)[r]{$1$}}
\put(67,30){\makebox(0,0)[l]{$1$}}

\end{picture}
}
$$
If we take $C_0$ in the proof of Lemma \ref{lem:adjust sum} $C_1$,
then $A_1=\{a_1\}$, $B_1=\{b_2\}$, $A_2=\{a_3\}$ and $B_2=\emptyset$.
Thus,
$$
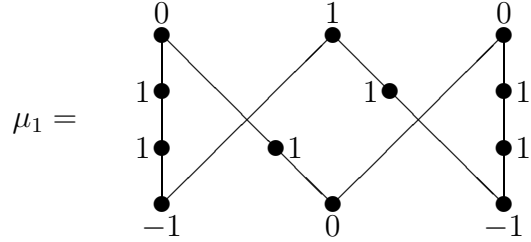
\mu_1=
\vcenter{\hsize=.4\textwidth\relax
\begin{picture}(55,50)
%\put(20,3){\makebox(0,0)[t]{$P_1\setminus\{x_0\}$}}

\put(5,10){\circle*{3}}
\put(5,20){\circle*{3}}
\put(5,30){\circle*{3}}
\put(5,40){\circle*{3}}

\put(35,10){\circle*{3}}
\put(35,40){\circle*{3}}

\put(65,10){\circle*{3}}
\put(65,20){\circle*{3}}
\put(65,30){\circle*{3}}
\put(65,40){\circle*{3}}
\put(25,20){\circle*{3}}
\put(45,30){\circle*{3}}

\put(5,10){\line(0,1){30}}
\put(65,10){\line(0,1){30}}
\put(5,10){\line(1,1){30}}
\put(5,40){\line(1,-1){30}}
\put(35,10){\line(1,1){30}}
\put(35,40){\line(1,-1){30}}

\put(5,8){\makebox(0,0)[t]{$-1$}}
\put(35,8){\makebox(0,0)[t]{$0$}}
\put(65,8){\makebox(0,0)[t]{$-1$}}
\put(5,42){\makebox(0,0)[b]{$0$}}
\put(35,42){\makebox(0,0)[b]{$1$}}
\put(65,42){\makebox(0,0)[b]{$0$}}
\put(3,20){\makebox(0,0)[r]{$1$}}
\put(27,20){\makebox(0,0)[l]{$1$}}
\put(67,20){\makebox(0,0)[l]{$1$}}
\put(3,30){\makebox(0,0)[r]{$1$}}
\put(43,30){\makebox(0,0)[r]{$1$}}
\put(67,30){\makebox(0,0)[l]{$1$}}

\end{picture}
}
$$
Now if we select $C_0$ to be $C_1$ again, then
 $A_1=\{a_1\}$, $B_1=\{b_2\}$, $A_2=\{a_3\}$ and $B_2=\emptyset$.
Thus,
$$
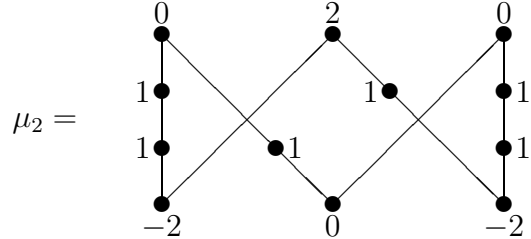
\mu_2=
\vcenter{\hsize=.4\textwidth\relax
\begin{picture}(55,50)

\put(5,10){\circle*{3}}
\put(5,20){\circle*{3}}
\put(5,30){\circle*{3}}
\put(5,40){\circle*{3}}

\put(35,10){\circle*{3}}
\put(35,40){\circle*{3}}

\put(65,10){\circle*{3}}
\put(65,20){\circle*{3}}
\put(65,30){\circle*{3}}
\put(65,40){\circle*{3}}
\put(25,20){\circle*{3}}
\put(45,30){\circle*{3}}

\put(5,10){\line(0,1){30}}
\put(65,10){\line(0,1){30}}
\put(5,10){\line(1,1){30}}
\put(5,40){\line(1,-1){30}}
\put(35,10){\line(1,1){30}}
\put(35,40){\line(1,-1){30}}

\put(5,8){\makebox(0,0)[t]{$-2$}}
\put(35,8){\makebox(0,0)[t]{$0$}}
\put(65,8){\makebox(0,0)[t]{$-2$}}
\put(5,42){\makebox(0,0)[b]{$0$}}
\put(35,42){\makebox(0,0)[b]{$2$}}
\put(65,42){\makebox(0,0)[b]{$0$}}
\put(3,20){\makebox(0,0)[r]{$1$}}
\put(27,20){\makebox(0,0)[l]{$1$}}
\put(67,20){\makebox(0,0)[l]{$1$}}
\put(3,30){\makebox(0,0)[r]{$1$}}
\put(43,30){\makebox(0,0)[r]{$1$}}
\put(67,30){\makebox(0,0)[l]{$1$}}

\end{picture}
}
$$
Finally, select $C_0$ to be $C_3$.
Then $A_1=\{a_2\}$, $B_1=\{b_3\}$, $A_2=\{a_3\}$ and $B_2=\emptyset$.
Therefore,
$$
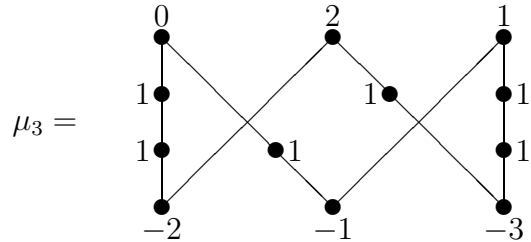
\mu_3=
\vcenter{\hsize=.4\textwidth\relax
\begin{picture}(55,50)
%\put(20,3){\makebox(0,0)[t]{$P_1\setminus\{x_0\}$}}

\put(5,10){\circle*{3}}
\put(5,20){\circle*{3}}
\put(5,30){\circle*{3}}
\put(5,40){\circle*{3}}

\put(35,10){\circle*{3}}
\put(35,40){\circle*{3}}

\put(65,10){\circle*{3}}
\put(65,20){\circle*{3}}
\put(65,30){\circle*{3}}
\put(65,40){\circle*{3}}
\put(25,20){\circle*{3}}
\put(45,30){\circle*{3}}

\put(5,10){\line(0,1){30}}
\put(65,10){\line(0,1){30}}
\put(5,10){\line(1,1){30}}
\put(5,40){\line(1,-1){30}}
\put(35,10){\line(1,1){30}}
\put(35,40){\line(1,-1){30}}

\put(5,8){\makebox(0,0)[t]{$-2$}}
\put(35,8){\makebox(0,0)[t]{$-1$}}
\put(65,8){\makebox(0,0)[t]{$-3$}}
\put(5,42){\makebox(0,0)[b]{$0$}}
\put(35,42){\makebox(0,0)[b]{$2$}}
\put(65,42){\makebox(0,0)[b]{$1$}}
\put(3,20){\makebox(0,0)[r]{$1$}}
\put(27,20){\makebox(0,0)[l]{$1$}}
\put(67,20){\makebox(0,0)[l]{$1$}}
\put(3,30){\makebox(0,0)[r]{$1$}}
\put(43,30){\makebox(0,0)[r]{$1$}}
\put(67,30){\makebox(0,0)[l]{$1$}}

\end{picture}
}
$$
and this function satisfies the desired properties.

Now we show the following.

\newcommand{\assumpast}{{\bf(*)}}
\begin{thm}
\mylabel{thm:chain p trace}
Let $\xi\in \SSSSS^{(0)}$ and $d=\xi(-\infty)$.
Then
$$
T^\xi\in\sqrt{\trace(\omega_{\kcp})}
$$
if and only if
\begin{enumerate}
\item
for any chain $C$ in $P$ with $\starcpx_P(C)$ is not pure,
$\xi^+(C)<d$
and
\item
\mylabel{item:cycle}
for any chains
$C_1$, \ldots, $C_u$, $C'_1$, \ldots, $C'_u$ in $P$ with
\begin{enumerate}
\item[\hskip\labelsep\assumpast]
$\max C_1<\min C'_1>\max C_2<\min C'_2>\cdots >\max C_u<\min C'_u>\max C_1$,
$\max C_i\not\leq\max C_j$ (resp.\ $\min C'_i\not\leq\min C'_j$) for any 
$i$ and $j$ with $i\neq j$,
$C_i$ (resp.\ $C'_i$) is a maximal chain in $(-\infty,\max C_i]$
(resp.\ $[\min C'_i,\infty)$),
and
$\sum_{i=1}^u\rank([\max C_i,\min C'_i])
>\sum_{i=1}^{u-1}\dist(\max C_{i+1},\min C'_i)
+\dist(\max C_1,\min C'_u)$,
\end{enumerate}
it holds that
$\sum_{i=1}^u(\xi^+(C_i)+\xi^+(C'_i))<ud$.
\end{enumerate}
\end{thm}
\begin{proof}
The ``only if'' part follows from Lemmas \ref{lem:nonpure link} and \ref{lem:cycle}.
Now we prove the ``if'' part.

First, note that \assumpb\ of Lemma \ref{lem:adjust sum} is satisfied. 
In fact, if there are chains 
$(C_1, \ldots , C_u, C'_1, \ldots, C'_u)$
which do not satisfy \assumpb\ of Lemma \ref{lem:adjust sum},
then $\sum_{i=1}^u(\xi^+(C_i)+\xi^+(C'_i))=ud$ and
$$
\sum_{i=1}^u\rank([\max C_i,\min C'_i])> 
\sum_{i=1}^{u-1}\dist(\max C_{i+1},\min C'_i)+\dist(\max C_1,\min C'_u)
$$
or
$$
\sum_{i=1}^u\rank([\max C_i,\min C'_i])<
\sum_{i=1}^{u-1}\dist(\max C_{i+1},\min C'_i)+\dist(\max C_1,\min C'_u).
$$
The former case contradicts \ref{item:cycle} 
and in the latter case, set of chains $(C_u,\ldots ,C_1,C'_{u-1},\ldots,C'_1,C'_u)$ 
violates the condition \ref{item:cycle}.

By Lemma \ref{lem:adjust sum}, we see that there exists $\mu\in\ZZZ^P$ such that
$\mu(z)=1$ for any $z\in P\setminus\supp\xi$ and $\mu^+(C)=\mu^+(C')$ for any
$C$, $C'\in C^{[d]}_\xi$.
Set $m\define\mu^+(C)$ for $C\in C^{[d]}_\xi$ if $C^{[d]}_\xi\neq\emptyset$ and
$m\define0$ if $C^{[d]}_\xi=\emptyset$.
Take a huge integer $N$ ($N>\sum_{z\in P}|\mu(z)|+|m|$)
and set
\begin{eqnarray*}
\eta(z)&=&
\left\{
\begin{array}{ll}
N\xi(z)+\mu(z),&\quad z\in P,\\
Nd+1+m,&\quad z=-\infty,
\end{array}
\right.
\\
\zeta(z)&=&
\left\{
\begin{array}{ll}
N\xi(z)-\mu(z),&\quad z\in P,\\
Nd-1-m,&\quad z=-\infty.
\end{array}
\right.
\end{eqnarray*}
Then $\eta(z)\geq1$ and $\zeta(z)\geq -1$ for any $z\in P$,
since $\mu(z)=1$ for any $z\in P\setminus\supp\xi$, 
$\xi(z)>0$ for any $z\in\supp\xi$
and $N$ is a huge integer.

Let $C$ be an arbitrary maximal chain in $P$.
If $C\not\in C^{[d]}_\xi$, then $\xi^+(C)<d$.
Thus,
$\eta^+(C)=N\xi^+(C)+\mu^+(C)<Nd+m=\eta(-\infty)-1$
and
$\zeta^+(C)=N\xi^+(C)-\mu^+(C)<Nd-m=\zeta(-\infty)+1$
since $N$ is a huge integer
($\mu^+(C)$ may not be equal to  $m$ in this case).
If $C\in C^{[d]}_\xi$, then
$\eta^+(C)=N\xi^+(C)+\mu^+(C)=Nd+m=\eta(-\infty)-1$
and
$\zeta^+(C)=N\xi^+(C)-\mu^+(C)=Nd-m=\zeta(-\infty)+1$.
Thus, we see that $\eta\in\SSSSS^{(1)}$ and
$\zeta\in\SSSSS^{(-1)}$.
Since $\eta+\zeta=2N\xi$, we see that
$(T^\xi)^{2N}\in\trace(\omega_{\kcp})$
by Fact \ref{fact:trace of an ideal}.
\end{proof}

%%%%%%%%%%%%%%%%%%%%%%%%%%%%%%%%%%%%%%%%%%

\section{The case of order polytopes}

\mylabel{sec:order}

In this section, we consider the trace
of the canonical module of the Ehrhart ring
$\kop$ of the order polytope $\msOOO(P)$ of a poset $P$.
By considering the set of edges of the Hasse diagram of a poset, 
we define a new poset from a given one.
%First we make the following.

\begin{definition}
\rm
\mylabel{def:cr}
Let $Q$ be a finite poset which is not an antichain.
We define the covering relation poset, denoted by $\CR(Q)$, of $Q$
as follows.
The base set of $\CR(Q)$ is
$\{(x,y)\mid x\covered y$ in $Q\}$
and for $(x_1, y_1)$, $(x_2, y_2)\in\CR(Q)$,
we define
$(x_1,y_1)<(x_2,y_2)$ in $\CR(Q)\iff y_1\leq x_2$ in $Q$.
\end{definition}

First we note the following fact.

\begin{lemma}
\mylabel{lem:cr corr}
There is  a one to one correspondence between the sets
$$
\FFFFF\define\{\nu\in\ZZZ^{P^\pm}\mid\nu(\infty)=0\}
$$
and
$$
\GGGGG\define
\left\{\xi\in\ZZZ^{\CR(P^\pm)}\left|\
\vcenter{\hsize=.5\textwidth\relax
\noindent
$\xi^+(C)=\xi^+(C')$ for any maximal chains $C$ and $C'$ in $\CR(P^\pm)$}
\right.\right\}
$$
by
$\Phi\colon\FFFFF\to\GGGGG$ and $\Psi\colon\GGGGG\to\FFFFF$
defined by
$\Phi(\nu)(x,y)=\nu(x)-\nu(y)$
for $\nu\in\FFFFF$ and $(x,y)\in\CR(P^\pm)$
and
$\Psi(\xi)(x)=\sum_{i=1}^t\xi(x_{i-1},x_i)$
for $\xi\in\GGGGG$ and $x\in P^-$,
where
$x=x_0\covered x_1\covered\cdots\covered x_{t-1}\covered x_t=\infty$.
\end{lemma}
Note that $\sum_{i=1}^t\xi(x_{i-1},x_i)$ is independent of the saturated chain
$x_0$, $x_1$, \ldots, $x_t$ from $x$ to $\infty$ since
$\xi^+(C)=\xi^+(C')$ for any maximal chains $C$ and $C'$ in $\CR(P^\pm)$.
Note also that
$\Psi(\xi_1+\xi_2)=\Psi(\xi_1)+\Psi(\xi_2)$.
Moreover,
$\nu(x)-\nu(y)\geq n$ for any $x$, $y\in P^\pm$ with $x\covered y$
 if and only if $\Phi(\nu)(\alpha)\geq n$ for any $\alpha\in\CR(P^\pm)$.

Next we note the following fact.

\begin{lemma}
\mylabel{lem:pure corr}
Let $Q$ be a poset, $(x_1,y_1)$, $(x_2, y_2)\in\CR(Q)$
and $(x_1,y_1)<(x_2,y_2)$ in $\CR(Q)$.
Then $\rank([(x_1,y_1),(x_2,y_2)]_{\CR(Q)})=\rank([y_1,x_2]_Q)+1$
and
$\dist_{\CR(Q)}((x_1,y_1),(x_2,y_2))=\dist_Q(y_1,x_2)+1$.
In particular,
$[(x_1,y_1),(x_2,y_2)]_{\CR(Q)}$ is pure if and only if
$[y_1,x_2]_Q$ is pure.
\end{lemma}

In order to prove the main theorem of this section, we note the following fact.

\begin{lemma}
\mylabel{lem:del eq}
Let $\nu\in\TTTTT^{(0)}$.
Assume that for any $a_1$, \ldots, $a_t$, $b_1$, \ldots, $b_t\in P^\pm$ with
$a_1<b_1>a_2<b_2>\cdots>a_t<b_t>a_1$,
$a_i\not\leq a_j$ (resp.\ $b_i\not\leq b_j$) 
for any $i$ and $j$ with $i\neq j$
and
$\sum_{i=1}^t\rank([a_i,b_i])>
\sum_{i=1}^{t-1}\dist(a_{i+1},b_i)+\dist(a_1,b_t)$,
it holds that 
$\sum_{i=1}^t\nu(a_i)>\sum_{i=1}^t\nu(b_i)$.
Then for any $c_1$, \ldots, $c_u$, $d_1$, \ldots, $d_u\in P^\pm$ with
$c_1\leq d_1\geq c_2\leq d_2\geq\cdots\geq c_u\leq d_u\geq c_1$
and
$\sum_{i=1}^u\rank([c_i,d_i])>
\sum_{i=1}^{u-1}\dist(c_{i+1},d_i)+\dist(c_1,d_u)$,
it holds that
$\sum_{i=1}^u\nu(c_i)>\sum_{i=1}^u\nu(d_i)$.
\end{lemma}
\begin{proof}
We prove by induction on $u$.
If $u=1$, then since 
$c_1\leq d_1$ and
$\rank([c_1,d_1])>\dist(c_1,d_1)$,
we see that $c_1<d_1$.
Thus, the result follows from the assumption.

Assume that $u>1$.
If $c_1<d_1>c_2<d_2>\cdots>c_u<d_u>c_1$, 
$c_i\not\leq c_j$ (resp.\ $d_i\not\leq d_j$)
for any $i$ and $j$ with $i\neq j$
then the result follows from
the assumption.
Suppose first that
$c_1<d_1>c_2<d_2>\cdots>c_u<d_u>c_1$ and
$c_{j_1}\leq c_{j_2}$ with $1\leq j_1<j_2\leq u$.
Then
$$c_1<d_1>\cdots>c_{j_1}<d_{j_2}>\cdots>c_u<d_u>c_1$$
and
$$c_{j_1}<d_{j_1}>\cdots>c_{j_2-1}<d_{j_2-1}>c_{j_1}$$
and
\begin{eqnarray*}
0&<&
\sum_{i=1}^u\rank([c_i,d_i])-\sum_{i=1}^{u-1}\dist(c_{i+1},d_i)
-\dist(c_1,d_u)\\
&\leq&
\sum_{i=1}^{j_1-1}\rank([c_i,d_i])
+\rank([c_{j_2},d_{j_2}])
+\sum_{i=j_2+1}^u\rank([c_i,d_i])\\
&&\qquad
-\sum_{i=1}^{j_1-1}\dist(c_{i+1},d_i)-\sum_{i=j_2}^{u-1}\dist(c_{i+1},d_i)
-\dist(c_1,d_u)\\
&&\qquad
+\sum_{i=j_1}^{j_2-1}\rank([c_i,d_i])
-\sum_{i=j_1}^{j_2-2}\dist(c_{i+1},d_i)
-\dist(c_{j_2},d_{j_2-1})\\
&&\qquad
+\rank([c_{j_1},c_{j_2}])
-\dist(c_{j_1},c_{j_2})\\
&\leq&
\sum_{i=1}^{j_1-1}\rank([c_i,d_i])+
\rank([c_{j_1},d_{j_2}])
+\sum_{i=j_2+1}^u\rank([c_i,d_i])\\
&&\qquad
-\sum_{i=1}^{j_1-1}\dist(c_{i+1},d_i)-\sum_{i=j_2}^{u-1}\dist(c_{i+1},d_i)
-\dist(c_1,d_u)\\
&&\qquad
+\sum_{i=j_1}^{j_2-1}\rank([c_i,d_i])
-\sum_{i=j_1}^{j_2-2}\dist(c_{i+1},d_i)
-\dist(c_{j_1},d_{j_2-1}).
\end{eqnarray*}
Thus,
\begin{eqnarray*}
0&<&\sum_{i=1}^{j_1-1}\rank([c_i,d_i])+
\rank([c_{j_1},d_{j_2}])
+\sum_{i=j_2+1}^u\rank([c_i,d_i])\\
&&\qquad
-\sum_{i=1}^{j_1-1}\dist(c_{i+1},d_i)-\sum_{i=j_2}^{u-1}\dist(c_{i+1},d_i)
-\dist(c_1,d_u)\\
\end{eqnarray*}
%\mbox{or}\\
or
\begin{eqnarray*}
0&<&
\sum_{i=j_1}^{j_2-1}\rank([c_i,d_i])
-\sum_{i=j_1}^{j_2-2}\dist(c_{i+1},d_i)
-\dist(c_{j_1},d_{j_2-1}).
\end{eqnarray*}
Since $u-(j_2-j_1)<u$ and $j_2-j_1<u$,
we see by induction hypothesis that
$$
\sum_{i=1}^{j_1-1}(\nu(c_i)-\nu(d_i))+
\sum_{i=j_2+1}^{u}(\nu(c_i)-\nu(d_i))>0
$$
or
$$
\sum_{i=j_1}^{j_2-1}(\nu(c_i)-\nu(d_i))>0.
$$
Since $\nu(c_i)\geq \nu(d_i)$ for any $i$, we see that
$$
\sum_{i=1}^u(\nu(c_i)-\nu(d_i))>0.
$$
The other cases are proved similarly.

Next suppose that $c_i=d_i$ with $1\leq i\leq u-1$.
set
\begin{eqnarray*}
c'_j&\define&
\left\{
\begin{array}{ll}
c_j,&\quad j\leq i-1,\\
c_{j+1},&\quad j\geq i,
\end{array}
\right.
\\
d'_j&\define&
\left\{
\begin{array}{ll}
d_j,&\quad j\leq i-1,\\
d_{j+1},&\quad j\geq i,
\end{array}
\right.
\end{eqnarray*}
for $1\leq j\leq u-1$.
Then
$c'_1\leq d'_1\geq c'_2\leq d'_2\geq\cdots \geq c'_{u-1}\leq d'_{u-1}\geq c'_1$.
Further,
\begin{eqnarray*}
\sum_{j=1}^{u-1}\rank([c'_j,d'_j])&=&
\sum_{j=1}^{i-1}\rank([c_j,d_j])+
\sum_{j=i+1}^{u}\rank([c_j,d_j])\\
&=&
\sum_{j=1}^{u}\rank([c_j,d_j])
\end{eqnarray*}
since $\rank([c_i,d_i])=0$.
On the other hand,
\begin{eqnarray*}
&&
\sum_{j=1}^{u-2}\dist(c'_{j+1},d'_j)+\dist(c'_1,d'_{u-1})\\
&=&
\sum_{j=1}^{i-2}\dist(c_{j+1},d_j)+
\dist(c_{i+1},d_{i-1})+
\sum_{j=i+1}^{u-1}\dist(c_{j+1},d_j)+
\dist(c_1,d_u)
\\
&\leq&
\sum_{j=1}^{i-2}\dist(c_{j+1},d_j)+
\dist(c_{i+1},d_{i})+
\dist(c_{i},d_{i-1})+
\sum_{j=i+1}^{u-1}\dist(c_{j+1},d_j)+
\dist(c_1,d_u)\\
&=&\sum_{j=1}^{u-1}\dist(c_{j+1},d_j)+\dist (c_1,d_u)
\end{eqnarray*}
since $c_i=d_i$.
Therefore,
$$
\sum_{j=1}^{u-1}\rank([c'_j,d'_j])>
\sum_{j=1}^{u-1}\dist(c'_{j+1},d'_j)+\dist(c'_1,d'_{u-1}).
$$
Thus, we see by induction hypothesis, that
$$
\sum_{j=1}^u\nu(c_j)-\sum_{j=1}^u\nu(d_j)
=
\sum_{j=1}^{u-1}\nu(c'_j)-\sum_{j=1}^{u-1}\nu(d'_j)>0.
$$

The cases where $i=u$, $c_{\ell+1}=d_\ell$ for some $\ell$ with
$1\leq \ell\leq u-1$ or $c_1=d_u$ are proved similarly.
\end{proof}

Now we state the following.

\newcommand{\assumpstar}{\mbox{\bf($\star$)}}
\begin{thm}
\mylabel{thm:order p trace}
Let $\nu\in\TTTTT^{(0)}$.
Then
$$
T^\nu\in\sqrt{\trace(\omega_{\kop})}
$$
if and only if for any $a_1$, \ldots, $a_u$,
$b_1$, \ldots, $b_u\in P^\pm$ with
\begin{enumerate}
\item[\hskip-\labelsep\assumpstar]
$a_1< b_1> a_2< b_2>\cdots> a_u< b_u> a_1$,
$a_i\not\leq a_j$ (resp.\ $b_i\not\leq b_j$) for any $i$ and $j$ with $i\neq j$
and
$\sum_{i=1}^u\rank([a_i,b_i])>
\sum_{i=1}^{u-1}\dist(a_{i+1},b_i)+\dist(a_1,b_u)$,
\end{enumerate}
it holds that
$\sum_{i=1}^u\nu(a_i)>\sum_{i=1}^u\nu(b_i)$.
\end{thm}
\begin{proof}
We first prove the ``only if'' part.
Since $T^\nu\in\sqrt{\trace(\omega_{\kop})}$,
there exist a positive integer $N$, $\eta\in\TTTTT^{(1)}$
and $\zeta\in\TTTTT^{(-1)}$ with $\eta+\zeta=N\nu$.
Since $\eta\in\TTTTT^{(1)}$, we see that 
$\eta(a_i)-\eta(b_i)\geq\rank([a_i,b_i])$ for $1\leq i\leq u$.
On the other hand, since $\zeta\in\TTTTT^{(-1)}$, we see that
$\zeta(a_{i+1})-\zeta(b_i)\geq-\dist(a_{i+1},b_i)$
for $1\leq i\leq u-1$ and
$\zeta(a_1)-\zeta(b_u)\geq-\dist(a_1,b_u)$.
Therefore,
\begin{eqnarray*}
&&N\left(\sum_{i=1}^u(\nu(a_i)-\nu(b_i))\right)\\
&=&
\sum_{i=1}^u(\eta(a_i)+\zeta(a_i)-\eta(b_i)-\zeta(b_i))\\
&\geq&
\sum_{i=1}^u\rank([a_i,b_i])-\sum_{i=1}^{u-1}\dist(a_{i+1},b_i)-\dist(a_1,b_u)\\
&>&0.
\end{eqnarray*}
Since $N>0$, we see that
%this contradicts to the assumption 
$\sum_{i=1}^u\nu(a_i)>\sum_{i=1}^u\nu(b_i)$.

Next we prove the ``if'' part.
We use symbols $\FFFFF$, $\GGGGG$, $\Phi$ and $\Psi$ of Lemma \ref{lem:cr corr}.
Let $\xi\define\Phi(\nu)$ and set $d\define\nu(-\infty)$.
Since $\nu\in\TTTTT^{(0)}$, we see that $\xi(\alpha)\geq 0$ for any
$\alpha\in\CR(P^\pm)$.
Thus, by setting $\xi(-\infty_{\CR(P^\pm)})=d$, we see that 
$\xi\in\SSSSS^{(0)}(\CR(P^\pm))$.

Let $C$ be a chain in $\CR(P^\pm)$ with $\xi^+(C)=d$.
We show that $\starcpx_{\CR(P^\pm)}(C)$ is pure.
Assume the contrary and set $C=\{(x_1,y_1), \ldots, (x_t,y_t)\}$,
$x_1$, \ldots, $x_t$, $y_1$, \ldots, $y_t\in P^\pm$,
$y_i\leq x_{i+1}$ in $P^\pm$ for $1\leq i\leq t-1$.
Then $\nu(x_1)=d$, $\nu(y_t)=0$ and $\nu(y_j)=\nu(x_{j+1})$ for $1\leq j\leq t-1$, since
\begin{eqnarray*}
d&=&\nu(-\infty)\\
&=&
\nu(-\infty)-\nu(x_1)+\sum_{j=1}^ t(\nu(x_j)-\nu(y_j))+\sum_{j=1}^ {t-1}(\nu(y_j)-\nu(x_{j+1}))
+\nu(y_t)\\
%-\nu(\infty)\\
&=&(\nu(-\infty)-\nu(x_1))+\xi^+(C)+\sum_{j=1}^ {t-1}(\nu(y_j)-\nu(x_{j+1}))
+\nu(y_t)\\
%-\nu(\infty)\\
\end{eqnarray*}
and
$\nu(y_j)\geq\nu(x_{j+1})$ for $1\leq j\leq t$,
$\nu(-\infty)\geq\nu(x_1)$,
$\nu(y_t)\geq 0$ and $\xi^+(C)=d$.

Since $\linkcpx_{\CR(P^\pm)}(C)$ is not pure, 
$\{\alpha\in\CR(P^\pm)\mid\alpha>(x_t,y_t)\}$
or
$\{\alpha\in\CR(P^\pm)\mid\alpha<(x_1,y_1)\}$
is not pure
or there is $i$ with $1\leq i\leq t-1$ such that
$[(x_i,y_i),(x_{i+1},y_{i+1})]_{\CR(P^\pm)}$ is not pure.
If $[(x_i,y_i),(x_{i+1},y_{i+1})]_{\CR(P^\pm)}$ is not pure,
we see by Lemma \ref{lem:pure corr} that 
$[y_i,x_{i+1}]_{P^\pm}$ is not pure.
Therefore, we see that $\rank([y_i,x_{i+1}])>\dist(y_i,x_{i+1})$
and by assumption that $\nu(y_i)>\nu(x_{i+1})$.
This contradicts to the fact shown above.
We can deduce a contradiction in other cases by the same way.

Next, let $C_1$, \ldots, $C_u$, $C'_1$, \ldots, $C'_u$ be chains
in $\CR(P^\pm)$ 
which satisfy \assumpastast\ 
of page \pageref{page:assumpastast}.
%of Lemma \ref{lem:adjust sum}.
Set $\max C_i\defines(x_i,y_i)$ and $\min C'_i\defines (w_i,z_i)$
for $1\leq i\leq u$.
Since
$y_1\leq w_1\geq y_2\leq w_2\geq\cdots\geq y_u\leq w_u\geq y_1$
and 
$\xi^+(C_i)+\xi^+(C'_j)=d$
for any $i$ and $j$, we see that
$\sum_{i=1}^u\nu(y_i)=\sum_{i=1}^u\nu(w_i)$
by the same reason above.
Thus, we see that
$$
\sum_{i=1}^u\rank([y_i,w_i]_{P^\pm})
=\sum_{i=1}^{u-1}\dist_{P^\pm}(y_{i+1}, w_i)+\dist_{P^\pm}(y_1,w_u)
$$
by Lemma \ref{lem:del eq}
and therefore, by Lemma \ref{lem:pure corr} that
\begin{eqnarray*}
&&
\sum_{i=1}^u\rank([\max C_i,\min C'_i]_{\CR(P^\pm)})\\
&=&
\sum_{i=1}^{u-1}\dist_{\CR(P^\pm)}(\max C_{i+1},\min C'_i)
+\dist_{\CR(P^\pm)}(\max C_1,\min C'_u).
\end{eqnarray*}
Therefore, we see by Lemma \ref{lem:adjust sum} that there exists $\mu\in\ZZZ^{\CR(P^\pm)}$
such that $\mu(\alpha)=1$ if $\xi(\alpha)=0$ and
$\mu^+(C)=\mu^+(C')$ for any maximal chains $C$ and $C'$ in $\CR(P^\pm)$.
(Note that for any maximal chain $C$ in $\CR(P^\pm)$,
$\xi^+(C)=\nu(-\infty)=d$.)

Set $m\define\mu^+(C)$, where $C$ is a maximal chain in $\CR(P^\pm)$.
Let $N$ be a huge integer and set
\begin{eqnarray*}
&&
\eta(\alpha)=N\xi(\alpha)+\mu(\alpha),\\
&&
\zeta(\alpha)=N\xi(\alpha)-\mu(\alpha)
\end{eqnarray*}
for $\alpha\in\CR(P^\pm)$.
Then $\eta(\alpha)\geq 1$, $\zeta(\alpha)\geq -1$ for any $\alpha\in\CR(P^\pm)$.
Further, for any maximal chain $C$ in $\CR(P^\pm)$,
\begin{eqnarray*}
&&
\eta^+(C)=N\xi^+(C)+\mu^+(C)=Nd+m,\\
\mbox{and}\\
&&\zeta^+(C)=N\xi^+(C)-\mu^+(C)=Nd-m.
\end{eqnarray*}
Thus, $\eta$ and $\zeta$ are elements of $\GGGGG$ of Lemma \ref{lem:cr corr}.
Let $\nu'\define\Psi(\eta)$ and $\nu''\define\Psi(\zeta)$.
Then, since $\eta(\alpha)\geq 1$ (resp.\ $\zeta(\alpha)\geq -1$)
for any $\alpha\in\CR(P^\pm)$, we see that
$\nu'\in\TTTTT^{(1)}$ (resp.\ $\nu''\in\TTTTT^{(-1)}$).
Further, since $\eta+\zeta=2N\xi$,
we see that
$
\nu'+\nu''=2N\nu$.
This means that
$$
(T^\nu)^{2N}\in\trace(\omega_{\kop}).
$$
\end{proof}

%%%%%%%%%%%%%%%%%%%%%%%%%%%%%%%%%%%%%%%%%%%%%%%

\section{Dimesions of non-\gor\ loci}

\mylabel{sec:dim}

In this section, we study the dimensions of  non-\gor\ loci
of the Ehrhart rings of order and chain polytopes.
We first consider the order polytopes.

For $a_1$, \ldots, $a_u$, $b_1$, \ldots, $b_u\in P^\pm$
satisfying \assumpstar\ of Theorem \ref{thm:order p trace}, set
$$
\TTTTT^{(0)}_{(a_1, \ldots, a_u, b_1, \ldots, b_u)}\define\{
\nu\in\TTTTT^{(0)}\mid \sum_{i=1}^u\nu(a_i)>\sum_{i=1}^u\nu(b_i)\}
$$
and
$$
\pppp'_{(a_1, \ldots, a_u,b_1,\ldots, b_u)}\define
\bigoplus_{\nu\in\TTTTT^{(0)}_{(a_1, \ldots, a_u,b_1, \ldots, b_u)}}\KKK T^\nu.
$$
Then $\pppp'_{(a_1, \ldots, a_u,b_1, \ldots, b_u)}$ is a prime ideal of $\kop$.
Moreover, we see by Theorem \ref{thm:order p trace} that
$$
\sqrt{\trace(\omega_{\kop})}=\bigcap_{(a_1, \ldots, a_u,b_1, \ldots, b_u)}
\pppp'_{(a_1, \ldots, a_u,b_1, \ldots, b_u)},
$$
where $(a_1, \ldots, a_u,b_1, \ldots, b_u)$ runs through sequence of elements
satisfying \assumpstar\ of Theorem \ref{thm:order p trace}.

For a sequence $(a_1, \ldots, a_u, b_1, \ldots, b_u)$
of elements in $P^\pm$ satisfying \assumpstar\ of Theorem \ref{thm:order p trace}, 
we set
$\mabu\define\{x\in P^\pm\mid$ there are $i$ and $j$ with $a_i\leq x\leq b_j\}$.
We first consider the case where $-\infty\not\in\{a_1,\ldots,a_u\}$
 and $\infty \not\in\{b_1,\ldots,b_u\}$.
In this case, 
for a sequence $(a_1, \ldots, a_u,b_1, \ldots, b_u)$ of
elements satisfying \assumpstar\ of Theorem \ref{thm:order p trace},
$\kop/\pppp'_{(a_1, \ldots, a_u,b_1, \ldots, b_u)}$ is isomorphic to the 
Ehrhart ring of the face 
$\gabu\define\{f\in\msOOO(P)\mid f(a_i)=f(b_j)$ for any $i$ and $j\}$ of 
$\msOOO(P)$.
Note that 
for $\nu\in\TTTTT^{(0)}$,
$\nu(a_i)\geq\nu(b_i)$ for any $1\leq i\leq u$,
$\nu(a_{i+1})\geq \nu(b_i)$ for any $1\leq i\leq u-1$ and $\nu(a_u)\geq \nu(b_1)$.
In particular, we see 
that if $\nu\in\TTTTT^{(0)}\setminus
\TTTTT^{(0)}_{(a_1, \ldots, a_u, b_1, \ldots, b_u)}$, then
$\nu(a_i)=\nu(b_j)$ for any $i$ and $j$,
%by induction on $|i-j|$,
since $\sum_{i=1}^u\nu(a_i)=
\sum_{i=1}^u\nu(b_i)$,
and therefore 
$\nu(x)=\nu(a_1)$ for any $x\in\mabu$.

Fix a sequence $a_1$, \ldots, $a_u$, $b_1$, \ldots, $b_u$ which satisfies
\assumpstar\ of Theorem \ref{thm:order p trace}.
Let $\ast$ be a new element not contained in $P$ and set
$\overline P\define(P\setminus\mabu)\cup\{\ast\}$.
We define the relation $<$ on $\overline P$ as follows.
For $\alpha$, $\beta\in\overline P$, $\alpha<\beta$ if and only if
at least one of the following 4 conditions is satisfied.
\begin{enumerate}
\item
$\alpha$, $\beta\in P$ and $\alpha<\beta$ in $P$.
\item
$\alpha\in P$, $\beta=\ast$ and there is 
$x\in\mabu$ with $\alpha <x$ in $P$.
\item
$\alpha=\ast$, $\beta\in P$ and there is 
$x\in\mabu$ with $x<\beta$ in $P$.
\item
\mylabel{item:beyond m}
$\alpha$, $\beta\in P$ and there are 
$x$, $y\in\mabu$ with
$\alpha<x$ and $y<\beta$
in $P$.
\end{enumerate}
Note that
$\overline P$ is not the quotient poset of $P$ (cf. \cite{sta4}).
However, we have the following.

\begin{lemma}
\mylabel{lem:bar p poset}
By the relation $<$ above, $\overline P$ is a poset.
\end{lemma}
\begin{proof}
First, let $\alpha\in\overline P$.
We will show that $\alpha\not<\alpha$.
If $\alpha=\ast$, it is clear from the definition of $<$ that $\alpha\not<\alpha$.
Suppose that $\alpha\in P$.
Then the only way for $\alpha < \alpha$ is by condition \ref{item:beyond m} above.
Since the existence of 
$x$, $y\in\mabu$ with $\alpha <x$ and $y<\alpha$ imply
$\alpha\in\mabu$, we see that $\alpha\not<\alpha$.

Next we show that $<$ is transitive.
Let $\alpha$, $\beta$, $\gamma\in \overline P$ and suppose that
$\alpha<\beta$ and $\beta<\gamma$ in $\overline P$.

If $\alpha<\beta$ in $P$, then it is clear that $\alpha<\gamma$
in $\overline P$.

Next suppose that $\beta=\ast$.
Then $\alpha$, $\gamma\in P$ and
there are 
$x$, $y\in\mabu$
such that $\alpha<x$ and $y<\gamma$ in $P^\pm$.
Therefore, $\alpha<\gamma$ in $\overline P$.

Suppose that $\alpha=\ast$.
Then $\beta\in P$ and there is 
$x\in\mabu$ with $x<\beta$ 
in $P^\pm$.
If $\gamma=\ast$, then there is 
$y\in\mabu$ with $\beta<y$ in $P^\pm$,
and it follows that $\beta\in\mabu$.
This is a contradiction.
If $\gamma\in P$ and there are 
$z$, $w\in\mabu$ with $\beta<z$ and $w<\gamma$  in $P^\pm$,
then it follows that 
$\beta\in\mabu$
again.
Therefore, $\beta$, $\gamma\in P$ and $\beta<\gamma$ in $P$.
It follows that $x<\gamma$ in 
$P^\pm$ and therefore, $\alpha<\gamma$ in $\overline P$.

Finally suppose that $\alpha$, $\beta\in P$ and there are 
$x$, $y\in\mabu$ with $\alpha<x$ and $y<\beta$ in $P^\pm$.
Then there is no 
$z\in\mabu$ with $\beta<z$ in $P^\pm$
by the same reason as above.
Therefore, $\gamma\in P$ and $\beta<\gamma$ in $P$.
Since $\alpha<x$ and $y<\gamma$ in 
$P^\pm$, we see that $\alpha<\gamma$
in $\overline P$.
\end{proof}

Next, we state the following lemma whose proof is easy.

\begin{lemma}
\mylabel{lem:aff iso}
Let $\Theta\colon\RRR^{\overline P}\to\RRR^P$ be a map defined by
$$
\Theta(f)(x)=
\left\{
\begin{array}{ll}
f(x),&\quad x\in P\setminus\mabu,\\
f(\ast),&\quad x\in\mabu.
\end{array}
\right.
$$
Then $\Theta$ is an injective linear map
and $\Theta(\msOOO(\overline P))=\gabu$.
In particular, 
$$
\dim\gabu=\#\overline P=\#P-\#\mabu+1
$$
and
$$
\dim 
\kop/\pppp'_{(a_1, \ldots, a_u,b_1, \ldots, b_u)}
=\#P-\#M+2.
$$
\end{lemma}

Next we consider the case where $-\infty\in\{a_1,\ldots, a_u\}$
and $\infty\not\in\{b_1,\ldots,b_u\}$.
Since $a_i\not\leq a_j$ for $i\neq j$, it follows that $u=1$.
We set $a\define a_1$ and $b\define b_1$.
Then $\kop/\pppp'_{(a,b)}$ is isomorphic to the Ehrhart ring
of the face
$\gab\define\{f\in\msOOO(P)\mid f(x)=1$ for any $x\in(-\infty,b]_{P^\pm}\}$
of $\msOOO(P)$.
Then $\gab$ is isomorphic to the order complex $\msOOO(P\setminus\mab)$
of $P\setminus \mab$
(here we make a convention that the order complex of an empty set is a point).
Further, since $a=-\infty\in\mab\setminus P$, we see that
$$
\dim\msOOO(P\setminus\mab)=\#(P\setminus\mab)=\#P-\#\mab+1
$$
and
$$
\kop/\pppp'_{(a,b)}
=\#P-\#\mab+2,
$$
i.e.,
$$
\kop/\pppp'_{(a_1,\ldots, a_u,b_1, \ldots,b_u)}
=\#P-\#\mabu+2.
$$

Similarly, if $-\infty\not\in\{a_1, \ldots, a_u\}$ and
$\infty\in\{b_1, \ldots, b_u\}$, 
we see that
$$
\kop/\pppp'_{(a_1,\ldots, a_u,b_1, \ldots,b_u)}
=\#P-\#\mabu+2.
$$

Finally, if $-\infty\in\{a_1, \ldots, a_u\}$ and
$\infty\in\{b_1, \ldots, b_u\}$, then
$\pppp'_{(a_1,\ldots, a_u,b_1, \ldots,b_u)}=\mmmm_{\kop}$
and $\mabu=P^\pm$.
Therefore,
$$
\kop/\pppp'_{(a_1,\ldots, a_u,b_1, \ldots,b_u)}
=0=\#P-\#\mabu+2.
$$

Thus, we see the following.

\begin{thm}
\mylabel{thm:order p nongor}
$\{\pppp\in\spec(\kop)\mid
\kop_{\pppp}$ is not \gor$\}$
is a closed subset of $\spec(\kop)$ with dimension
$$
\max_{(a_1, \ldots, a_u,b_1,\ldots, b_u)}\#P-\#\mabu+2
$$
where $(a_1,\ldots,a_u,b_1,\ldots,b_u)$ runs through the
sequence of elements of $P^\pm$ which satisfy \assumpstar\ of 
Theorem \ref{thm:order p trace}.
\end{thm}

\begin{cor}\label{cor:anysize}
Suppose that $\kop$ is not \gor, i.e., $P$ is not pure.
Then the dimension of the non-\gor\ locus of $\kop$ is at most $\dim\kop-4$.
In particular, for any $\pppp\in\spec(\kop)$ with $\height\pppp\leq 3$,
$\kop_\pppp$ is \gor.
Conversely,
let $m$ and $n$ be integers with $0 \leq m \leq n-4$.  
Then there exists a Hibi ring $\kop$ of dimension $n$ with a non-\gor\ locus of dimension $m$.
\end{cor}
\begin{proof}
We first prove the first part.
Since $\dim\kop=\#P+1$ and $\#\mabu\geq5$ for any 
sequence $(a_1,\ldots,a_u,b_1,\ldots,b_u)$ 
of elements of $P^\pm$ which satisfy \assumpstar\ of 
Theorem \ref{thm:order p trace},
we see by Theorem \ref{thm:order p nongor} that the dimension of non-\gor\ locus
of $\kop$ is at most $\dim\kop-4$.

Conversely, let $m$ and $n$ be integers with $0\leq m\leq n-4$.
Let $P_1$ be a totally ordered set with $\#P_1=n-m-2$
and let $P_2$ be a single point with $P_1 \cap P_2 = \emptyset$.
Let $C$ be a totally ordered set with $\#C=m$.
Set $P=(P_1+P_2)\oplus C$ (cf. \cite[Section 3.2]{sta5}),
i.e., for $x$, $y\in P$, $x<y$ if and only if
(i) $x$, $y\in P_1$ and $x<y\in P_1$,
(ii) $x$, $y\in C$ and $x<y\in C$ or
(iii) $x\in P_1\cup P_2$ and $y\in C$.
Note that $P$ is not pure since $\#P_1=n-m-2 \geq 2$.  
Then $\#P=n-1$ and therefore
$\kop$ has dimension $n$ its non-\gor\ locus has dimension 
$m$, which comes from $(-\infty,\min C)$.
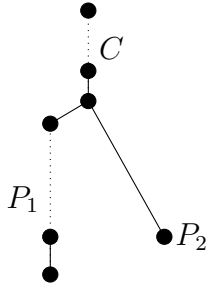
\begin{center}
\begin{tikzpicture}
\draw[dotted] (.5,2.3) -- (.5,3.5);
\draw (.5,2.3) -- (.5,2.7);
\draw[dotted] (0,0) -- (0,2);
\draw(0,0) -- (0,.5);
\draw(.5,2.3) -- (0,2);
\draw (.5,2.3) -- (1.5,.5);
\draw[fill] (.5,2.7) circle [radius=0.1];
\draw[fill] (.5,3.5) circle [radius=0.1];
\draw[fill] (0,.5) circle [radius=0.1];
\draw[fill] (.5,2.3) circle [radius=0.1];
\draw[fill] (0,0) circle [radius=0.1];
\draw[fill] (0,2) circle [radius=0.1];
\draw[fill] (1.5,.5) circle [radius=0.1];
\node [left] at (0,1) {$P_1$};
\node [right] at (1.5,.5) {$P_2$};
\node [right] at (.5,3) {$C$};
\end{tikzpicture}
\end{center}
\end{proof}

Next, we consider the chain polytopes.

For a chain $C$ with $\starcpx_P(C)$ is not pure, we set
$$
\SSSSS^{(0)}_C\define\{\xi\in\SSSSS^{(0)}\mid
\xi^+(C)<\xi(-\infty)\}
$$
and
$$
\pppp_C\define\bigoplus_{\xi\in\SSSSS^{(0)}_C}\KKK T^\xi.
$$
For chains $C_1$, \ldots, $C_u$, $C'_1$, \ldots, $C'_u$ in $P$
satisfying \assumpast\ of \ref{item:cycle} of Theorem \ref{thm:chain p trace},
we set
$$
\SSSSS^{(0)}_{(C_1, \ldots, C_u,C'_1, \ldots, C'_u)}
\define\{\xi\in\SSSSS^{(0)}\mid
\sum_{i=1}^u(\xi^+(C_i)+\xi^+(C'_i))<u\xi(-\infty)\}
$$
and
$$
\pppp_{(C_1,\ldots, C_u,C'_1,\ldots, C'_u)}\define
\bigoplus_{\xi\in\SSSSS^{(0)}_{(C_1,\ldots, C_u,C'_1,\ldots, C'_u)}}\KKK T^\xi.
$$
Then $\pppp_C$ and $\pppp_{(C_1, \ldots, C_u, C'_1, \ldots, C'_u)}$ are prime ideals of
$\kcp$.
Moreover, we see by Theorem \ref{thm:chain p trace} that
$$
\sqrt{\trace(\omega_{\kcp})}=
\bigcap_{C}\pppp_C\cap\bigcap_{(C_1,\ldots, C_u,C'_1,\ldots, C'_u)}
\pppp_{(C_1,\ldots, C_u,C'_1,\ldots, C'_u)},
$$
where $C$ runs through chains in $P$ with $\starcpx_P(C)$ is not pure
and $(C_1, \ldots, C_u,C'_1,\ldots, C'_u)$ runs through the sets of chains 
satisfying \assumpast\ of \ref{item:cycle} of Theorem \ref{thm:chain p trace}.
Therefore,
\begin{eqnarray*}
&&
\dim \kcp/\trace(\omega_{\kcp})\\
&=&
\max\{\max_{C}\coht\pppp_C,
\max_{(C_1,\ldots, C_u,C'_1,\ldots, C'_u)}\coht\pppp_{(C_1,\ldots, C_u, C'_1,\ldots, C'_u)}\}.
\end{eqnarray*}

Note that for a nonempty chain $C$ in $P$,
$\kcp/\pppp_C$ is isomorphic to the Ehrhart ring of the face
$\fc\define\{f\in\msCCC(P)\mid f^+(C)=1\}$ of $\msCCC(P)$ and
for set of chains $(C_1, \ldots, C_u, C'_1, \ldots, C'_u)$ 
satisfying \assumpast\ of \ref{item:cycle} of Theorem \ref{thm:chain p trace},
$\kcp/\pppp_{(C_1,\ldots, C_u,C'_1,\ldots, C'_u)}$
is isomorphic to the Ehrhart ring of the face
$\fcu\define\{f\in\msCCC(P)\mid \sum_{i=1}^u(f^+(C_i)+f^+(C'_i))=u\}$
of $\msCCC(P)$.
For such $(C_1$, \ldots, $C_u$, $C'_1$, \ldots, $C'_u)$, we set
$\lcu\define\{x\in P\mid$ there are $i$ and $j$ such that $\max C_i\leq x\leq\min C'_j\}$.

\begin{lemma}
\mylabel{lem:face dim}
\begin{enumerate}
\item
\mylabel{item:dim chain}
For any nonempty chain $C$ in $P$, it holds that
$$\dim\fc=\#P-\#\linkcpx_P(C)-1.$$
\item
\mylabel{item:dim cycle}
For set of chains $(C_1$, \ldots, $C_u$, $C'_1$, \ldots, $C'_u)$ 
satisfying \assumpast\ of \ref{item:cycle} of Theorem \ref{thm:chain p trace},
it holds that
$$
\dim(\fcu\cap\RRR^{\lcu})=1
$$
and therefore
$$\dim\fcu\leq\#P-\#\lcu+1.$$
\end{enumerate}
\end{lemma}
\begin{proof}
\ref{item:dim chain}
First note that for any $f\in\fc$ and $x\in\linkcpx_P(C)$,
it holds that $f(x)=0$.
Next let $x$ be an arbitrary element of $P\setminus\starcpx_P(C)$.
Take $y\in C$ with $y\not\sim x$.
Then it is easily verified that
$$
f_1(z)=
\left\{
\begin{array}{ll}
1,&\quad z=x\mbox{ or }y,\\
0,&\quad\mbox{otherwise,}
\end{array}
\right.
$$
and
$$
f_2(z)=
\left\{
\begin{array}{ll}
1,&\quad z=y,\\
0,&\quad\mbox{otherwise,}
\end{array}
\right.
$$
are elements of $\fc$.
Therefore, the vector subspace of $\RRR^P$ parallel to $\aff\fc$ 
contains
$\chi_{\{x\}}=f_1-f_2$.

On the other hand, the image of $\fc$ under the projection
$\RRR^P\to\RRR^C$, $f\mapsto f|_C$ is
$\{g\in\RRR^C\mid g(c)\geq 0$ for any $c\in C$ and $g^+(C)=1\}$,
which is a $\#C-1$ dimensional polytope.
Therefore,
$$
\dim\fc=\#P-\#\starcpx_P(C)+(\#C-1)=\#P-\#\linkcpx_P(C)-1.
$$

\ref{item:dim cycle}
Set $L\define \lcu$, $\msFFF\define\fcu$,
$a_i\define\max C_i$ and $b_i\define\min C'_i$ for $1\leq i\leq u$,
$A\define\{a_1$, \ldots, $a_u\}$, $B\define\{b_1$, \ldots, $b_u\}$
and $L'\define L\setminus (A\cup B)$.

First note that $f^+(C_i)+f^+(C'_j)=1$ for any $f\in\msFFF$, $i$ and $j$.
Therefore, for any $x\in L'$ and $f\in\msFFF$, $f(x)=0$.
Thus, 
$$
\aff\msFFF\cap\RRR^L=
\left\{f\in\RRR^L\left|\ \vcenter{\hsize=.5\textwidth\relax\noindent
$f(x)=0$ for any $x\in P\setminus(A\cup B)$ and there are $\alpha$
and $\beta$ with $\alpha+\beta=1$ and $f(a_i)=\alpha$, $f(b_i)=\beta$
for any $i$}\right.\right\}.
$$
It follows that 
$\dim(\aff\msFFF\cap\RRR^L)=1$.
Since $\aff\msFFF\cap\RRR^L\neq\emptyset$, the dimension of
$\aff\msFFF$ is the sum of dimensions of $\aff\msFFF\cap\RRR^L$
and the image of $\aff\msFFF$ by the projection $\RRR^P\to\RRR^{P\setminus L}$.
Thus, we see that
$$
\dim\msFFF\leq \#(P\setminus L)+1=\#P-\#L+1.
$$
\end{proof}

There is an example that the equality in the inequation
in \ref{item:dim cycle} of Lemma \ref{lem:face dim} does not hold.

\begin{example}
\rm
\mylabel{ex:not right dim}
Let
$$
P=
\vcenter{\hsize=.4\textwidth\relax
\begin{picture}(35,50)
%\put(20,3){\makebox(0,0)[t]{$P_1\setminus\{x_0\}$}}

\put(5,10){\circle*{3}}
\put(5,20){\circle*{3}}
\put(5,30){\circle*{3}}
\put(5,40){\circle*{3}}
\put(25,10){\circle*{3}}
\put(25,20){\circle*{3}}
\put(25,40){\circle*{3}}

\put(5,10){\line(0,1){30}}
\put(5,20){\line(1,1){20}}
\put(5,40){\line(1,-1){20}}
\put(25,10){\line(0,1){30}}
\put(5,20){\line(2,-1){20}}
\put(5,10){\line(2,1){20}}

\put(3,10){\makebox(0,0)[r]{$d_1$}}
\put(3,20){\makebox(0,0)[r]{$a_1$}}
\put(3,30){\makebox(0,0)[r]{$e_{\ }$}}
\put(3,40){\makebox(0,0)[r]{$b_1$}}
\put(27,10){\makebox(0,0)[l]{$d_2$}}
\put(27,20){\makebox(0,0)[l]{$a_2$}}
\put(27,40){\makebox(0,0)[l]{$b_2$}}

\end{picture}
}
$$
and set $C_1\define\{a_1, d_1\}$, 
$C_2\define\{a_2, d_2\}$, 
$C'_1\define\{b_1\}$ and 
$C'_2\define\{b_2\}$.
Then
$$
\msFFF_{(C_1,C_2,C'_1,C'_2)}=
\left\{f\in\RRR^P\left|\ \vcenter{\hsize=.5\textwidth\relax\noindent
$f(e)=0$, there are $\alpha$, $\beta$ and $\gamma$ with
$\alpha\geq0$, $\beta\geq 0$, $\gamma\geq 0$ , $\alpha+\beta+\gamma=1$,
$f(a_i)=\alpha$, $f(b_i)=\beta$ and $f(d_i)=\gamma$ for $i=1$, $2$}
\right.\right\}.
$$
In fact, the right hand side is clearly contained in $\msFFF_{(C_1,C_2,C'_1,C'_2)}$.
Suppose that $f\in\msFFF_{(C_1,C_2,C'_1,C'_2)}$.
Then
\begin{eqnarray}
&&f(d_1)+f(a_1)+f(b_1)=1
\nonumber\\
&\mbox{and}&
\mylabel{eq:=1}\\
&&f(d_2)+f(a_2)+f(b_2)=1.
\nonumber
\end{eqnarray}
Further,
\begin{eqnarray*}
&&f(e)\geq0\\
&&f(d_1)+f(a_1)+f(e)+f(b_1)\leq 1\\
&&f(d_1)+f(a_1)+f(b_2)\leq 1\\
&&f(d_2)+f(a_2)+f(b_1)\leq 1\\
&&f(d_2)+f(a_1)+f(b_1)\leq 1\\
&&f(d_1)+f(a_2)+f(b_2)\leq 1.
\end{eqnarray*}
Therefore, 
by combining the first and the second inequlities and equation \refeq{eq:=1},
we see that 
equalities hold for the first 2 inequations.
We also see that equalities hold for the third and the fourth inequalities
by combining them and equation \refeq{eq:=1}
and
that equalities hold for the fifth and the sixth inequalities
by combining them and equation \refeq{eq:=1}.
%equalities hold for the above 6 inequations.
Thus, $f(e)=0$, $f(a_1)=f(a_2)$,
$f(b_1)=f(b_2)$ and
$f(d_1)=f(d_2)$
and we see that $f$ is contained in the right hand side.

Therefore, $\dim\msFFF_{(C_1,C_2,C'_1,C'_2)}=2$,
while $\#P-\#L+1=7-5+1=3$.
\end{example}

Although the equality does not hold in general in \ref{item:dim cycle} of 
Lemma \ref{lem:face dim}, we can replace chains to appropriate ones
so that equality in \ref{item:dim cycle} of Lemma \ref{lem:face dim} holds
and $\max C_i$ (resp.\ $\min C'_i$) are the same as the original ones.
For example, if we replace $C_2$ to $\{a_2, d_1\}$ in the above example,
then $\dim\msFFF_{(C_1,C_2,C'_1,C'_2)}=3$.

In order to prove this fact in general, we state the following.

\begin{lemma}
\mylabel{lem:chains const}
Let 
$a_1$, \ldots, $a_u$, $b_1$, \ldots, $b_u$ be elements of $P$
satisfying \assumpstar\ of Theorem \ref{thm:order p trace} and
$u\geq 2$.
Then there is a set $(C_1, \ldots, C_u,C'_1, \ldots, C'_u)$
of chains in $P$ with
$\max C_i=a_i$ (resp.\ $\min C'_i=b_i$) for $1\leq i\leq u$,
satisfying \assumpast\ of \ref{item:cycle} of Theorem \ref{thm:chain p trace}
and
$$
\dim\fcu=\#P-\#\lcu+1.
$$
\end{lemma}
\begin{proof}
Set $M\define\mabu$,
$A\define\{a_1, \ldots, a_u\}$, $B\define\{b_1, \ldots, b_u\}$,
$D_1\define\{x\in P\mid x<a_i$ for some $i\}$,
$D_2\define\{x\in P\mid x>b_i$ for some $i\}$,
$D_3\define P\setminus(M\cup D_1\cup D_2)$
and
$t_i\define\height a_i$ (resp.\ $t'_i\define\coht b_i$) for $1\leq i\leq u$.
Set also $D_1=\{d_1, d_2, \ldots, d_v\}$
so that
$\height d_1\geq\height d_2\geq\cdots\geq \height d_v$
(resp.\ $D_2=\{d'_1, d'_2, \ldots, d'_{v'}\}$
so that
$\coht d'_1\geq\coht d'_2\geq\cdots\geq \coht d'_{v'}$).
We define $C_i$ inductively for each $i$.
Set $c_{i0}\define a_i$.
If $c_{i0}$, $c_{i1}$, \ldots, $c_{is}$, $s<t_i$, are defined so that
$c_{i0}>c_{i1}>\cdots>c_{is}$ and
$\height c_{ij}=t_i-j$ for $0\leq j\leq s$,
we set
$\ell(i,s+1)\define\min\{\ell\mid d_\ell<c_{is}\}$
and
$c_{i,s+1}\define d_{\ell(i,s+1)}$.
Then we
%\begin{trivlist}
%\item[\hskip\labelsep\bf Claim]
\begin{claim}
$\height c_{i,s+1}=t_i-s-1$.
\end{claim}
%\end{trivlist}
%
%\begin{proof}
In fact,
since $c_{i,s+1}<c_{is}$ and $\height c_{is}=t_i-s$, we see that
$\height c_{i,s+1}\leq t_i-s-1$.
Take an element $x\in P$ with $x<c_{is}$ and $\height x=t_i-s-1$.
Then $x\in D_1$.
Therefore, $x=d_{\ell'}$ for some $\ell'$.
Since $\ell(i,s+1)=\min\{\ell\mid d_\ell<c_{is}\}$,
we see that $\ell(i,s+1)\leq\ell'$.
Therefore, we see that
$\height c_{i,s+1}=\height d_{\ell(i,s+1)}\geq \height d_{\ell'}=\height x=t_i-s-1$.
%\end{proof}

\medskip

We define $C_i\define\{c_{i0}, c_{i1}, \ldots, c_{it_i}\}$ for $1\leq i\leq u$.
We define $C'_i$ for $1\leq i\leq u$ similarly.
Then it is clear by the definition that
$\lcu=M$.
Further, since
$a_1$, \ldots, $a_u$, $b_1$, \ldots, $b_u$
satisfy \assumpstar\ of Theorem \ref{thm:order p trace},
we see that
$(C_1, \ldots, C_u,C'_1, \ldots, C'_u)$
satisfy \assumpast\ of \ref{item:cycle} of Theorem \ref{thm:chain p trace}.

Set $D'_1\define\bigcup_{i=1}^u C_i\setminus A$ and
$D'_1=\{d_{j(1)}, d_{j(2)}, \ldots, d_{j(w)}\}$,
$j(1)<j(2)<\cdots<j(w)$.
We define subsets $A_0$, $A_1$, \ldots, $A_w$ of $P$ inductively.
We set $A_0\define A$.
If $A_0$, $A_1$, \ldots, $A_t$, $t<w$ are defined, we set
$A_{t+1}\define(A_t\cup\{d_{j(t+1)}\})\setminus\{x\in A_t\mid x>d_{j(t+1)}\}$.
We also set $D'_{1t}\define\{d_{j(1)}, d_{j(2)}, \ldots, d_{j(t)}\}$ for 
$1\leq t\leq w$ and $D'_{10}\define\emptyset$.
Then we have the following.

\begin{claim}
\mylabel{lem:ind lemma}
For $0\leq t\leq w$, it holds that
\begin{enumerate}
\item
\mylabel{item:subset chain}
$A_t\subset\bigcup_{i=1}^u C_i$.
\item
\mylabel{item:intersect}
$A_t\cap C_i\neq\emptyset$ for any $i$.
\item
\mylabel{item:antichain}
$A_t$ is an antichain.
\item
\mylabel{item:not<future}
$k>t$, $a\in A_t \Rightarrow a\not< d_{j(k)}$.
\item
\mylabel{item:not<prev}
$0\leq t'\leq t''\leq t$, $a'\in A_{t'}$, $a''\in A_{t''}\Rightarrow a'\not<a''$.
\item
\mylabel{item:subset}
$A_t\subset D'_{1t}\cup M$.
\end{enumerate}
\end{claim}
%
%\begin{proof}
We prove this claim by induction on $t$.
The case where $t=0$ is obvious.

Suppose that $t>0$.
First we see \ref{item:subset chain} from the construction of $A_t$ and the induction 
hypothesis.
Next we prove \ref{item:intersect}.
Since $A_{t-1}\cap C_i\neq\emptyset$ by induction hypothesis, take $a'\in A_{t-1}\cap C_i$.
If $a'\not>d_{j(t)}$, then $a'\in A_t$ and therefore $A_t\cap C_i\neq\emptyset$.
Suppose that $a'>d_{j(t)}$.
Let $a'=c_{is}$ and $\ell(i,s+1)\define\min\{\ell\mid d_\ell<c_{is}\}$.
Since $d_{j(t)}<a'=c_{is}$, we see that $j(t)\geq \ell(i,s+1)$.
Since $d_{\ell(i,s+1)}=c_{i,s+1}\in D'_1$, there is $t'$ with $\ell(i,s+1)=j(t')$.
Since $j(t)\geq j(t')$, we see that $t\geq t'$.
Suppose that $t'\leq t-1$.
Then $d_{j(t')}\in A_{t'}$, $a'\in A_{t-1}$ and 
$d_{j(t')}=d_{\ell(i,s+1)}=c_{i,s+1}<c_{is}=a'$,
contradicting \ref{item:not<prev} of the induction hypothesis.
Therefore, $t'=t$ and 
$d_{j(t)}=c_{i,s+1}\in A_t\cap C_i$.

Next we prove \ref{item:antichain}.
Let $a$, $a'\in A_t$ and $a\neq a'$.
If $a$, $a'\in A_{t-1}$, then by the induction hypothesis, we see that
$a\not\sim a'$.
Suppose that $a'=d_{j(t)}$.
Then $a\in A_{t-1}$ and by the definition of $A_t$, we see that $a'=d_{j(t)}\not<a$.
Further, by \ref{item:not<future} of the induction hypothesis, we see that
$a\not<d_{j(t)}=a'$.

Next we prove \ref{item:not<future}.
Suppose that $a\in A_t$ and $k>t$.
If $a\in A_{t-1}$, then by the induction hypothesis, we see that $a\not<d_{j(k)}$.
Thus, we assume that $a=d_{j(t)}$.
Since $k>t$, we see that $j(k)>j(t)$.
Therefore, $\height d_{j(k)}\leq\height d_{j(t)}$.
Thus, $d_{j(k)}\not>d_{j(t)}=a$.

Next we prove \ref{item:not<prev}.
Suppose that $0\leq t'\leq t''\leq t$, $a'\in A_{t'}$ and $a''\in A_{t''}$.
If $t'=t$, then, since $A_t$ is an antichain, we see that $a'\not< a''$.
Therefore, we may assume that $t'\leq t-1$.
If $a''\in A_{t-1}$, then by the induction hypothesis, we see that $a'\not< a''$.
Thus we may assume that $t''=t$ and $a''=d_{j(t)}$.
Then by \ref{item:not<future} of the induction hypotehsis, we see that
$a'\not <d_{j(t)}=a''$, since $t'<t$.

Finally, we prove \ref{item:subset}.
By induction hypothesis, we see that 
$A_{t-1}\subset D'_{1,t-1}\cup M$.
Therefore,
$A_t\subset A_{t-1}\cup\{d_{j(t)}\}\subset D'_{1,t}\cup M$.
%\end{proof}

Thus, we have proved Claim \ref{lem:ind lemma}

\medskip

For $0\leq t\leq w$, we set 
$$
f_t\define\frac{1}{2}\chi^P_{(B\cup A_t)}.
$$
Then we

\begin{claim}
$f_t\in\fcu$ for $0\leq t\leq w$.
\end{claim}
%
%\begin{proof}
In fact, let $C$ be an arbitrary maximal chain in $P$.
Then
$$
f_t^+(C)=\frac{1}{2}(\#(C\cap(B\cup A_t)))
=\frac{1}{2}(\#(C\cap B)+\#(C\cup A_t)).
$$
Since $B$ and $A_t$ are antichains, we see that $\#(C\cap B)\leq 1$
and $\#(C\cap A_t)\leq 1$.
Therefore, $f^+(C)\leq 1$.
Moreover, we see by \ref{item:intersect} and \ref{item:antichain} of 
Claim \ref{lem:ind lemma} that
$f_t^+(C_i)=\frac{1}{2}$ for any $i$.
Further, $f_t^+(C'_i)=\frac{1}{2}$ by the definition of $C'_i$.
Therefore, 
$f_t^+(C_i)+f_t^+(C'_i)=1$
for $1\leq i\leq u$.
%\end{proof}

\medskip

By \ref{item:subset} of Claim \ref{lem:ind lemma} and the fact that
$d_{j(t)}\in A_t$, we see that 
$A_t\setminus(D'_{1,t-1}\cup M)=\{d_{j(t)}\}$.
Therefore,
$$
(f_t-f_{t-1})|_{P\setminus(D'_{1,t-1}\cup M)}=
\frac{1}{2}\chi^{P\setminus(D'_{1,t-1}\cup M)}_{\{d_{j(t)}\}}
$$
for $1\leq t\leq w$.
Let $W$ be 
the vector subspace of $R^{P\setminus M}$ which is parallel to the
image of $\aff\fcu$ by the projection
$\RRR^P\to\RRR^{P\setminus M}$.
Then, we see by the above argument that $\chi^{P\setminus M}_{\{x\}}\in W$
for any $x\in D'_1$.

Next, let $x\in D_1\setminus D'_1$.
Set $A'(x)\define\{z\in P\mid z\in D'_1, \height z=\height x$ or
$z\in A, \height z\leq \height x$ or
$z=x\}$.
Then $A'(x)$ is an antichain, since $A$ is an antichain.
Also set $f=\frac{1}{2}\chi^P_{(B\cup A'(x))}$.
Then we can show the following.
\begin{claim}
$f\in\fcu$.
\end{claim}
%
%\begin{proof}
In fact, for any maximal chain $C$ in $P$, $f^+(C)\leq 1$ since $B$ and
$A'(x)$ are antichains.
Further, for any $i$ with $1\leq i\leq u$, it is easily verified that
$C_i\cap A'(x)\neq\emptyset$.
Thus, $f^+(C_i)=\frac{1}{2}$  for $1\leq i\leq u$.
Further, we see that $f^+(C'_i)=\frac{1}{2}$ for $1\leq i\leq u$
by the definition of $C'_i$.
%\end{proof}

\medskip

Since the image of $f-f_0$ by the projection $\RRR^P\to\RRR^{P\setminus (M\cup D'_1)}$
is $\frac{1}{2}\chi^{P\setminus (M\cup D'_1)}_{\{x\}}$, we see that 
$\chi^{P\setminus M}_{\{x\}}\in W$ for any
$x\in D_1\setminus D'_1$.
Thus, we see that for any $x\in D_1$, $\chi^{P\setminus M}_{\{x\}}\in W$.
We see that $\chi^{P\setminus M}_{\{x\}}\in W$ for any $x\in D_2$ by the same way.

Finally, let $x\in D_3$.
Set $f=\frac{1}{2}\chi^P_{(A\cup B\cup\{x\})}$.
Let $C$ be an arbitrary maximal chain in $P$.
If $f^+(C)>1$, then $x\in C$ and there are $a\in A\cap C$ and $b\in B\cap C$.
Since $x\not\in D_1\cup D_2$, it follows that $a\leq x\leq b$ and therefore
$x\in M$.
This contradicts to the definition of $D_3$.
Thus $f^+(C)\leq 1$.
Further, $f^+(C_i)=f^+(C'_i)=\frac{1}{2}$ by the definition
of $C_i$ and $C'_i$ for any $i$,
since $x\not \in D_1\cup D_2$.
Therefore, $f\in\fcu$ and $\chi^{P\setminus M}_{\{x\}}$, the image of $2(f-f_0)$ by the
projection $\RRR^P\to\RRR^{P\setminus M}$,
is an element of $W$.
Thus, we see that $W=\RRR^{P\setminus M}$.

Since $\lcu=M$,
we see by \ref{item:dim cycle} of Lemma \ref{lem:face dim} that
$$
\dim\fcu=\#P-\#\lcu+1.
$$
\end{proof}

Let us observe the construction of Lemma \ref{lem:chains const} by an example.

\begin{example}
\rm
Let
$$
P=
\vcenter{\hsize=.4\textwidth\relax
\begin{picture}(70,70)
%\put(20,3){\makebox(0,0)[t]{$P_1\setminus\{x_0\}$}}

\put(5,10){\circle*{3}}
\put(5,20){\circle*{3}}
\put(5,30){\circle*{3}}
\put(5,40){\circle*{3}}
\put(5,50){\circle*{3}}
\put(5,60){\circle*{3}}

\put(25,10){\circle*{3}}
\put(25,20){\circle*{3}}
\put(25,30){\circle*{3}}
\put(25,40){\circle*{3}}
\put(25,60){\circle*{3}}

\put(45,30){\circle*{3}}
\put(45,40){\circle*{3}}
\put(45,60){\circle*{3}}

\put(65,30){\circle*{3}}
\put(65,40){\circle*{3}}
\put(65,60){\circle*{3}}

\put(5,10){\line(0,1){50}}
\put(25,10){\line(0,1){50}}
\put(45,30){\line(0,1){30}}
\put(65,30){\line(0,1){30}}

\put(5,10){\line(2,1){20}}
\put(5,20){\line(2,1){40}}
\put(5,40){\line(3,1){60}}
\put(5,20){\line(2,-1){20}}
\put(5,30){\line(2,-1){20}}
\put(5,60){\line(1,-1){20}}

\put(25,40){\line(2,-1){20}}
\put(25,60){\line(1,-1){20}}
\put(45,30){\line(2,1){20}}
\put(45,40){\line(2,-1){20}}
\put(45,60){\line(1,-1){20}}

\put(5,62){\makebox(0,0)[b]{$b_1$}}
\put(25,62){\makebox(0,0)[b]{$b_2$}}
\put(45,62){\makebox(0,0)[b]{$b_3$}}
\put(65,62){\makebox(0,0)[b]{$b_4$}}

\put(3,50){\makebox(0,0)[r]{$e$}}

\put(3,40){\makebox(0,0)[r]{$a_1$}}
\put(23,40){\makebox(0,0)[r]{$a_2$}}
\put(43,40){\makebox(0,0)[r]{$a_3$}}
\put(63,40){\makebox(0,0)[r]{$a_4$}}

\put(3,30){\makebox(0,0)[r]{$d_1$}}
\put(27,30){\makebox(0,0)[l]{$d_2$}}
\put(43,30){\makebox(0,0)[r]{$d_7$}}
\put(67,30){\makebox(0,0)[l]{$d_8$}}
\put(3,20){\makebox(0,0)[r]{$d_3$}}
\put(27,20){\makebox(0,0)[l]{$d_4$}}
\put(3,10){\makebox(0,0)[r]{$d_5$}}
\put(27,10){\makebox(0,0)[l]{$d_6$}}

\end{picture}
}
$$
Then $C_1=\{a_1,d_1,d_3, d_5\}$,
$C_2=\{a_2,d_2,d_3, d_5\}$,
$C_3=\{a_3,d_2,d_3, d_5\}$,
$C_4=\{a_4,d_7\}$,
$C'_1=\{b_1\}$,
$C'_2=\{b_2\}$,
$C'_3=\{b_3\}$ and
$C'_4=\{b_4\}$.
$D'_1=\{d_1,d_2,d_3, d_5, d_7\}$,
$j(1)=1$, $j(2)=2$, $j(3)=3$, $j(4)=5$ and $j(5)=7$.
$A_1=\{a_2, a_3, a_4, d_1\}$,
$A_2=\{a_4, d_1, d_2\}$,
$A_3=\{a_4, d_3\}$,
$A_4=\{a_4, d_5\}$ and
$A_5=\{d_5, d_7\}$.
$A'(d_4)=\{d_3, a_4, d_4\}$,
$A'(d_6)=\{d_5, d_7, d_6\}$ and
$A'(d_8)=\{d_5, d_7, d_8\}$.
\end{example}

Let
$(C''_1, \ldots, C''_u, C'''_1, \ldots, C'''_u)$
be a set of chains which satisfies
\assumpast\ of \ref{item:cycle} of Theorem \ref{thm:chain p trace}.
We set $a_i\define\max C''_i$, $b_i\define \min C'''_i$ for $1\leq i\leq u$.
Then $a_1$, \ldots, $a_u$, $b_1$, \ldots, $b_u$ are elements of $P$
which satisfy \assumpstar\ of Theorem \ref{thm:order p trace}.
By Lemma \ref{lem:chains const}, we see that 
there is a set $(C_1, \ldots, C_u,C'_1, \ldots, C'_u)$
of chains in $P$ with
$\max C_i=a_i$ (resp.\ $\min C'_i=b_i$) for $1\leq i\leq u$,
satisfying \assumpast\ of \ref{item:cycle} of Theorem \ref{thm:chain p trace}
and
$$
\dim\fcu=\#P-\#\lcu+1.
$$
Since $L_{(C''_1, \ldots, C''_u, C'''_1, \ldots, C'''_u)}=\lcu$,
we see, by 
%\ref{item:dim chain} of 
Lemma \ref{lem:face dim},
the following.

\begin{thm}
\mylabel{thm:chain p nongor}
$\{\pppp\in\spec(\kcp)\mid
\kcp_{\pppp}$ is not \gor$\}$
is a closed subset of $\spec(\kcp)$ with dimension
$$
\max\{
\max_C(\#P-\#\linkcpx_P(C)),
\max_{(C_1,\ldots, C_u, C'_1, \ldots, C'_u)}(\#P-\#\lcu+2)\},
$$
where $C$ runs through chains with $\starcpx_P(C)$ is not pure and
$(C_1,\ldots,C_u,C'_1, \ldots, C'_u)$ runs through 
the sets of chains 
satisfying \assumpast\ of \ref{item:cycle} of Theorem \ref{thm:chain p trace}.
\end{thm}

Finally, we show that the dimensions of non-\gor\ loci of the Ehrhart rings
of chain and order polytopes of a poset are the same.
First note that as is noted in the proof of Theorem \ref{thm:chain gor} that
$\kcp$ is \gor\ if and only if so is $\kop$.
Moreover, by Theorem \ref{thm:pspec gor} and 
%Proposition \ref{prop:pspec gor},
\cite[Corollary 3.5]{hmp},
we see that $\kcp$ is \gor\ on the punctured spectrum if and only if
so is $\kop$, if and only if $P$ is a disjoint union of disconnected pure posets.

Now consider the other case.
Then 
there is a nonpure connected component of $P$.
In other words, there are $a_1$, \ldots, $a_u$, $b_1$, \ldots, $b_u\in P^\pm$
which satisfy 
\assumpstar\ of Theorem \ref{thm:order p trace}.
%\ref{item:u=1} or \ref{item:u>1} of Lemma \ref{lem:not so red}.

Let $a_1$, \ldots, $a_u$, $b_1$, \ldots, $b_u$ be elements of $P^\pm$
which satisfy 
\assumpstar\ of Theorem \ref{thm:order p trace},
%\ref{item:u=1} or \ref{item:u>1} of Lemma \ref{lem:not so red}.
First consider the case where $u\geq 2$.
Then $-\infty\not\in\{a_1, \ldots, a_u\}$ and $\infty\not\in\{b_1, \ldots, b_u\}$,
i.e.,
$a_1, \ldots, a_u, b_1, \ldots, b_u\in P$,
since $a_i\not\leq a_j$ (resp.\ $b_i\not\leq b_j$) if $i\neq j$.

Let $(C_1, \ldots, C_u, C'_1, \ldots, C'_u)$ be a set of chains in $P$ 
constructed by Lemma \ref{lem:chains const}.
Then $\lcu=\mabu$ and therefore by Lemmas \ref{lem:aff iso} and \ref{lem:chains const},
we see that
$$
\coht\pppp_{(C_1, \ldots, C_u, C'_1, \ldots, C'_u)}=
\coht\pppp'_{(a_1, \ldots, a_u, b_1, \ldots, b_u)}.
$$

Next consider the case where $u=1$.
If $a$, $b\in P^\pm$, $a<b$ and $\rank([a,b])>\dist(a,b)$, then take a maximal chain $C_1$
(resp. $C'_1$) in $(-\infty,a]$
(resp.\ $[b,\infty)$)
and set $C=C_1\cup C'_1$
($C_1=\emptyset$ (resp.\ $C'_1=\emptyset$) if $a=-\infty$ (resp.\ $b=\infty$)).
Then $\starcpx_P(C)$ is not pure and $\linkcpx_P(C)=(a,b)$.
Thus, 
$\#\linkcpx_P(C)=\#M_{(a,b)}-2$ and therefore
by Lemmas \ref{lem:aff iso} and \ref{lem:face dim}, we see that
$$
\coht\pppp_{C}=\coht\pppp'_{(a,b)}.
$$
Thus, we have shown that $\dim(\kop/\trace(\omega_{\kop}))\leq\dim(\kcp/\trace(\omega_{\kcp}))$.

Conversely, assume that $u>1$ and chains $C_1$, \ldots, $C_u$, $C'_1$, \ldots, $C'_u$ in $P$
which satisfy \assumpast\ of \ref{item:cycle} of Theorem \ref{thm:chain p trace} are given.
Set $a_i\define\max C_i$ (resp. $b_i\define\min C'_i$) for $1\leq i\leq u$.
Then $a_1$, \ldots, $a_u$, $b_1$, \ldots, $b_u$ satisfy \assumpstar\ of 
Theorem \ref{thm:order p trace}.
Further, since $\lcu=\mabu$, we see by Lemma \ref{lem:face dim} that
$$
\coht\pppp_{(C_1, \ldots, C_u,C'_1, \ldots, C'_u)}
\leq
\coht\pppp'_{(a_1, \ldots, a_u,b_1, \ldots, b_u)}.
$$

Next assume that a chain $C$ in $P$ with $\starcpx_P(C)$ is not pure is given.
Set $C=\{c_1, \ldots, c_t\}$, $c_1<\cdots<c_t$.
Then at least one of $(-\infty,c_1]$, $[c_1,c_2]$, \ldots, $[c_{t-1},c_t]$,
$[c_t,\infty)$ is not pure.
Suppose that $[c_i, c_{i+1}]$ is not pure.
Then $\linkcpx_P(C)\supset(c_i,c_{i+1})=M_{(c_i,c_{i+1})}\setminus\{c_i,c_{i+1}\}$.
Therefore,
$$
\coht\pppp_{C}\leq\coht\pppp'_{(c_i,c_{i+1})}
$$
by Lemmas \ref{lem:aff iso} and \ref{lem:face dim}.

The cases where $(-\infty,c_1]$ is not pure or $[c_t,\infty)$ is not pure are
treated by similar ways.
Thus, we have shown that $\dim(\kcp/\trace(\omega_{\kcp}))\leq\dim(\kop/\trace(\omega_{\kop}))$.

Therefore, we see the following.

\begin{thm}
\mylabel{thm:dim same}
The non-\gor\ loci of the Ehrhart rings of the chain and the order polytopes of 
a poset have the same dimension.
\end{thm}

By the above argument and Corollary \ref{cor:anysize}, we see the following.

\begin{cor}
Suppose that $\kcp$ is not \gor, i.e., $P$ is not pure.
Then the dimension of the non-\gor\ locus of $\kcp$ is at most $\dim\kcp-4$.
In particular, for any $\pppp\in\spec(\kcp)$ with $\height\pppp\leq 3$,
$\kcp_\pppp$ is \gor.
Conversely,
let $m$ and $n$ be integers with $0 \leq m \leq n-4$.  
Then there exists a poset $P$ such that $\kcp$ has
 dimension $n$ with a non-\gor\ locus of dimension $m$.
\end{cor}

\begin{example}
\rm
Let 
$$
P=
\vcenter{\hsize=.4\textwidth\relax
\begin{picture}(35,40)
%\put(20,3){\makebox(0,0)[t]{$P_1\setminus\{x_0\}$}}

\put(5,10){\circle*{3}}
\put(5,20){\circle*{3}}
\put(5,30){\circle*{3}}
\put(25,10){\circle*{3}}
\put(25,30){\circle*{3}}

\put(5,10){\line(0,1){20}}
\put(5,10){\line(1,1){20}}
\put(5,30){\line(1,-1){20}}
\put(25,10){\line(0,1){20}}

\put(3,30){\makebox(0,0)[r]{$b_1$}}
\put(3,10){\makebox(0,0)[r]{$a_1$}}
\put(27,30){\makebox(0,0)[l]{$b_2$}}
\put(27,10){\makebox(0,0)[l]{$a_2$}}

\end{picture}
}
$$
Then
$$
\sqrt{\trace(\omega_{\kcp})}=
\pppp_{\{a_1\}}\cap\pppp_{\{b_1\}}\cap\pppp_{(\{a_1\},\{a_2\},\{b_1\},\{b_2\})},
$$
$$
\sqrt{\trace(\omega_{\kop})}=
\pppp'_{(a_1,\infty)}\cap\pppp'_{(-\infty,b_1)}\cap\pppp'_{(a_1,a_2,b_1,b_2)}
$$
and
$$
\dim\kcp/\trace(\omega_{\kcp})=\dim\kop/\trace(\omega_{\kop})=2.
$$

Let
$$
P=
\vcenter{\hsize=.4\textwidth\relax
\begin{picture}(35,50)
%\put(20,3){\makebox(0,0)[t]{$P_1\setminus\{x_0\}$}}

\put(5,10){\circle*{3}}
\put(5,20){\circle*{3}}
\put(5,30){\circle*{3}}
\put(5,40){\circle*{3}}
\put(25,10){\circle*{3}}
\put(25,30){\circle*{3}}

\put(5,10){\line(0,1){30}}
\put(5,20){\line(2,1){20}}
\put(5,20){\line(2,-1){20}}

\put(3,20){\makebox(0,0)[r]{$a$}}
\put(5,8){\makebox(0,0)[t]{$c$}}
\put(25,8){\makebox(0,0)[t]{$d$}}

\end{picture}
}
$$
Then
$$
\sqrt{\trace(\omega_{\kcp})}=\pppp_{\{a,c\}}\cap\pppp_{\{a,d\}},
\qquad
\sqrt{\trace(\omega_{\kop})}=\pppp'_{(a,\infty)}
$$
and
$$
\dim\kcp/\trace(\omega_{\kcp})=\dim\kop/\trace(\omega_{\kop})=3.
$$
Note that $\sqrt{\trace(\omega_{\kop})}$ is a prime ideal while
$\sqrt{\trace(\omega_{\kcp})}$ is not.

Let 
$$
P=
\vcenter{\hsize=.4\textwidth\relax
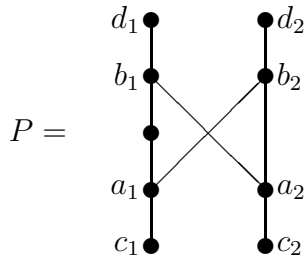
\begin{picture}(35,60)
%\put(20,3){\makebox(0,0)[t]{$P_1\setminus\{x_0\}$}}

\put(5,10){\circle*{3}}
\put(5,20){\circle*{3}}
\put(5,30){\circle*{3}}
\put(5,40){\circle*{3}}
\put(5,50){\circle*{3}}
\put(25,10){\circle*{3}}
\put(25,20){\circle*{3}}
\put(25,40){\circle*{3}}
\put(25,50){\circle*{3}}

\put(5,10){\line(0,1){40}}
\put(5,20){\line(1,1){20}}
\put(5,40){\line(1,-1){20}}
\put(25,10){\line(0,1){40}}

\put(3,10){\makebox(0,0)[r]{$c_1$}}
\put(3,20){\makebox(0,0)[r]{$a_1$}}
\put(3,40){\makebox(0,0)[r]{$b_1$}}
\put(3,50){\makebox(0,0)[r]{$d_1$}}
\put(27,10){\makebox(0,0)[l]{$c_2$}}
\put(27,20){\makebox(0,0)[l]{$a_2$}}
\put(27,40){\makebox(0,0)[l]{$b_2$}}
\put(27,50){\makebox(0,0)[l]{$d_2$}}

\end{picture}
}
$$
Then
$$
\sqrt{\trace(\omega_{\kcp})}=
\pppp_{\{a_1,c_1\}}\cap\pppp_{\{b_1,d_1\}}\cap
\pppp_{(\{a_1,c_1\},\{a_2,c_2\},\{b_1,d_1\},\{b_2,d_2\})},
$$
$$
\sqrt{\trace(\omega_{\kop})}=
\pppp'_{(a_1,\infty)}\cap\pppp'_{(-\infty,b_1)}\cap\pppp'_{(a_1,a_2,b_1,b_2)}
$$
and
$$
\dim\kcp/\trace(\omega_{\kcp})=\dim\kop/\trace(\omega_{\kop})=6.
$$
\end{example}

%===========================

\footnotesize{%
\noindent{$^{1}$ \textsc{Department of Mathematics, Kyoto University of Education,
1 Fujinomori, Fukakusa, Fushimi-ku, Kyoto, 612-8522, Japan}}
\\ \textit{E-mail address}: {\tt g53448@kyokyo-u.ac.jp}\\
\noindent{$^{2}$ \textsc{Department of Mathematics, University of Michigan, Ann Arbor, MI 48109, USA}}
\\ \textit{E-mail address}: {\tt jrpage@umich.edu}}


\begin{thebibliography}{HMP}
%
%
%
%
%
%
\bibitem[HHS]{hhs}
Herzog, J., Hibi, T. and Stamate, D. I.:
{\it The trace of the canonical module.}
Isr. J. Math. {\bf233} (2019), 133-165. 
https://doi.org/10.1007/s11856-019-1898-y
%arXiv:1612.02723v2.

\bibitem[HMP]{hmp}
Herzog, J., Mohammadi, F. and Page, J.:
{\it Measuring the non-Gorenstein locus of Hibi rings and normal affine semigroup rings.}
Journal of Algebra, {\bf540} (2019), 78-99. 


\bibitem[Hib]{hib}
Hibi, T.:
{\it Distributive lattices, affine smigroup rings and algebras 
with straightening laws.}
in ``Commutative Algebra and Combinatorics'' (M. Nagata and H. Matsumura, ed.),
Advanced Studies in Pure Math. {\bf11} North-Holland, Amsterdam (1987),
93--109.


\bibitem[HL]{hl}
Hibi, T., Li, N.:
{\it Chain Polytopes and Algebras with Straightening Laws.}
Acta Math Vietnam {\bf40}, 447--452 (2015). 
https://doi.org/10.1007/s40306-015-0115-2
%
%
\bibitem[Hoc]{hoc}
Hochster, M.:
{\it Rings of invariants of tori, Cohen-Macaulay rings generated
by monomials and polytopes.}
Ann. of Math. {\bf 96} (1972), 318-337.
%


%
%




%
%
%
%
\bibitem[Miy1]{mo}
%Mitsuhiro 
Miyazaki, M.:
{\it On the generators of the canonical module of a Hibi ring: A criterion of level property and the degrees of generators.}
Journal of Algebra {\bf480} (2017), 215-236.
%
%
\bibitem[Miy2]{mc}
Miyazaki, M.:
{\it
On the canonical ideal of the Ehrhart ring of the chain polytope of a poset.}
Journal of Algebra
{\bf541}, (Jan.\ 2020), 1--34.
%
%
\bibitem[Miy3]{mf}
Miyazaki, M.:
{\it 
Fiber cones, analytic spreads of the canonical and anticanonical ideals and limit Frobenius complexity of Hibi rings.}
to appear in
Journal of the Mathematical Society of Japan.
Advance publication (2020), 33 pages.
https://doi.org/10.2969/jmsj/81418141
%
%
%
\bibitem[Pag]{pag1}
Page, J.:
{\it The Frobenius Complexity of Hibi Rings}
Journal of Pure and Applied Algebra {\bf 223(2)} (2019) 580--604

\bibitem[Sta1]{sta2}
Stanley, R. P.:
{\it Hilbert Functions of Graded Algebras.}
Adv. Math. {\bf 28} (1978), 57--83.
%
\bibitem[Sta2]{sta3}
Stanley, R. P.:
{\it Two poset polytopes.} Discrete \& Computational Geometry 1.1 (1986), 9--23.

\bibitem[Sta3]{sta4}
Stanley, R. P.:
{\it Some applications of algebra to combinatorics.}
Discrete Applied Mathematics, 34(1-3),
 (1991), 241--277.

\bibitem[Sta4]{sta5}
Stanley, R. P.:
Enumerative Combinatorics Volume 1 second edition. 
Cambridge studies in advanced mathematics. (2011).
%
%
%
%
%
\end{thebibliography}
\end{document}